\documentclass[a4paper,10pt]{article}

\usepackage{geometry}
\usepackage{amsmath}
\usepackage{amsfonts,color,amssymb}
\usepackage{amsthm}
\usepackage{amssymb}
\usepackage{url}
\usepackage{todonotes}
\usepackage[T1]{fontenc}

\usepackage{mathrsfs} 

\usepackage{graphicx}
\usepackage{wrapfig}

\usepackage{tikz}
\usepackage{caption}
\usepackage{subcaption}

\numberwithin{equation}{section}

    \newtheorem{theo}{Theorem}\numberwithin{theo}{section}
    \newtheorem{coro}[theo]{Corollary}

    \newtheorem{lemma}[theo]{Lemma}

		\theoremstyle{remark}

    \def\N{\mathbb{N}}
    \def\E{\mathbb{E}}
    \def\0{{\bf 0}}

    \def\bdm{\begin{displaymath}}
    \newcommand{\edm}{\end{displaymath}}
    \def\benu{\begin{enumerate}}
    \def\eenu{\end{enumerate}}
    \def\beqn{\begin{equation}}
    \def\eeqn{\end{equation}}
    \def\be{\begin{equation}}
    \def\ee{\end{equation}}
    \def\bea{\begin{eqnarray}}
    \def\eea{\end{eqnarray}}
    \newcommand{\bean}{\begin{eqnarray*}}
    \newcommand{\eean}{\end{eqnarray*}}
    \newcommand{\bear}{\begin{eqnarray}}
    \newcommand{\eear}{\end{eqnarray}}

    \def\R{\mathbb{R}}

	\def\dint{\textup{d}}

    \def\qed{\hfill\hbox{${\vcenter{\vbox{
        \hrule height 0.4pt\hbox{\vrule width 0.4pt height 6pt
        \kern5pt\vrule width 0.4pt}\hrule height 0.4pt}}}$}}

\usepackage{titlesec}

\DeclareRobustCommand{\VAN}[3]{#2} 

\titleformat*{\section}{\normalfont\large\bfseries}
\titleformat*{\subsection}{\normalfont\bfseries}

\date{\vspace{-0.95cm}}

\begin{document}

\title{First passage percolation on {E}rd\H{o}s-{R}\'{e}nyi graphs with general weights}

\author{Fraser Daly\footnotemark[1], Matthias Schulte\footnotemark[2], and Seva Shneer\footnotemark[3]}


\maketitle

\footnotetext[1]{Heriot-Watt University, Edinburgh, UK. Email: 
    f.daly@hw.ac.uk}

\footnotetext[2]{Hamburg University of Technology, Germany. Email: matthias.schulte@tuhh.de}

\footnotetext[3]{Heriot-Watt University, Edinburgh, UK. Email: 
    v.shneer@hw.ac.uk}
		

\noindent{\bf Abstract} 
We consider first passage percolation on the Erd\H{o}s--R\'{e}nyi graph with $n$ vertices in which each pair of distinct vertices is connected independently by an edge with probability $\lambda/n$ for some $\lambda>1$. The edges of the graph are given non-negative i.i.d.\ weights with a non-degenerate distribution such that the probability of zero is not too large. We consider the paths with small total weight between two distinct typical vertices and analyse the joint behaviour of the numbers of edges on such paths, the so-called hopcounts, and the total weights of these paths. For $n\to\infty$, we show that, after a suitable transformation, the pairs of hopcounts and total weights of these paths converge in distribution to a Cox process, i.e., a Poisson process with a random intensity measure. The random intensity measure is controlled by two independent random variables, whose distribution is the solution of a distributional fixed point equation and is related to branching processes. For non-arithmetic and arithmetic edge-weight distributions we observe different behaviour. In particular, we derive the limiting distribution for the minimal total weight and the corresponding hopcount(s). Our results generalise earlier work of Bhamidi, van der Hofstad and Hooghiemstra, who assume that edge weights have an absolutely continuous distribution. The main tool we employ is the method of moments.
\vspace{12pt}


\vspace{12pt}


\section{Introduction, notation and main results}

First passage percolation is a classical and fundamental topic in applied probability. The basic model, as appearing in the original work of Hammersley and Welsh \cite{hw65}, is a simple one: given an underlying connected graph with $n$ vertices, each edge is assigned an independent and identically distributed weight. For two distinct vertices chosen uniformly at random we then consider quantities such as the number of edges of the minimum weight path between these vertices (i.e., the hopcount) and the total weight of such a path. This was originally motivated by modelling fluid flow, passage of information or population growth in a random medium, in which applications the hopcount and minimal total weight are of clear importance; see, for example, the discussion of modelling internet traffic using first passage percolation on an Erd\H{o}s--R\'{e}nyi graph in \cite{hhm01}. First passage percolation has also found applications in the analysis of interacting particle systems \cite{d88,l85}, epidemic models \cite{bhk14} and the Vickrey--Clarke--Groves auction mechanism \cite{fgs11}, among other areas.

Following the original work of Hammersley and Welsh \cite{hw65}, much of the work on first passage percolation has been in the setting where the underlying graph is a multi-dimensional integer lattice (see, for example, the survey \cite{adh17} and references therein). Many other deterministic and random graph structures have also been analysed, including the hypercube \cite{fp93,KSS}, the complete graph \cite{Bhamidi08,bh12,eghn20,Janson99}, nilpotent Cayley graphs \cite{bt15}, Cartesian power graphs \cite{m18}, Erd\H{o}s--R\'{e}nyi graphs \cite{Bhamidi08,bhh11,hhm01}, inhomogeneous random graphs \cite{bhh14,kk15}, the configuration model \cite{bhh10,bhh10_err,bhh17}, small-world networks \cite{kv16} and spatial random graphs \cite{clhjv23,hdgs15}, among others. In many of these papers the edge weights are assumed to be exponentially distributed. This assumption allows for a technically more straightforward analysis, exploiting Markovian structure in the graph exploration process. Exceptions include the works of Bhamidi \emph{et al}$.$ \cite{bhh14,bhh17}, where the authors generalise their earlier work on the configuration model \cite{bhh10,bhh10_err} to allow general non-negative weights with an absolutely continuous distribution. As noted by Bhamidi \emph{et al}$.$, in many applications the edge-weight distribution is unknown and so this is an important relaxation.      

Our focus in this paper is on the Erd\H{o}s--R\'{e}nyi graph. We allow for general non-negative i.i.d.\ edge weights with a non-degenerate distribution, which does not need to have a density, and derive the joint limiting distribution of the hopcounts and total weights of paths with small total weight between two distinct uniformly chosen vertices as $n\to\infty$, thereby generalising Theorems 2.1 and 2.2 of \cite{bhh11} and Theorem 3.3 of \cite{bhh14}. We use the remainder of this section to more formally introduce the model with which we are working, before stating our main results (Theorems \ref{thm:main} and \ref{thm:shortest_path} below), whose proofs are given in Section \ref{sec:proof}.

For a fixed $\lambda>1$ consider the supercritical Erd\H{o}s--R\'{e}nyi graph $G(n,\lambda/n)$ with vertex set $[n]=\{1,\ldots,n\}$ with $n\in\mathbb{N}$ and $n\geq \lambda$, and where each pair of distinct vertices is connected independently by an edge with probability $\lambda/n$. To each of the edges we associate a random weight. We assume throughout that these edge weights are independent and identically distributed, almost surely non-negative, and not almost surely constant. Weights are assigned independently of the underlying Erd\H{o}s--R\'enyi graph. We denote the distribution of a weight by $F$ and let $X$ be a random variable with this distribution. 

The idea of first passage percolation is to study the paths with small total weight between two distinct typical vertices, where typical means that the two vertices are chosen uniformly and independently from all pairs of distinct vertices. By the symmetry properties of the Erd\H{o}s--R\'{e}nyi graph, it does not make a difference whether we choose two vertices at random or whether we take two fixed vertices. Therefore, we always consider paths between the vertices $1$ and $n$ in the following. We require that the vertices of a path are distinct. To this end, we let $\mathcal{P}_1^n$ be the collection of paths in $G(n,\lambda/n)$ between vertices $1$ and $n$. For a path $p \in \mathcal{P}_1^n$ we let $H(p)$ denote its hopcount (i.e., the total number of edges) and write $L(p)$ for its total weight (i.e., the sum of the weights of all involved edges). We note that the notation $L$ is traditionally used in the literature and refers to length of paths. We keep the notation for ease of comparison but refer to the relevant characteristic as total weight as we find that length may in principle refer to either the total weight or the total number of edges. In the work that follows we will study the collection $(H(p),L(p))_{p\in\mathcal{P}_1^n}$ of pairs of hopcounts and total weights of paths from $1$ to $n$. Formally one can think of them as a point process in $\mathbb{R}^2$. We will show that this point process converges after a suitable transformation to a Cox process as $n\to\infty$. This describes the limiting joint behaviour of all hopcounts and total weights of paths with small total weight.

Note that we require the distribution of the weights to be non-degenerate, i.e., that the weights are not almost surely constant. For such degenerate weights, the total weight is a multiple of the hopcount. The limiting distribution of the graph distance between two typical vertices of the Erd\H{o}s--R\'{e}nyi graph, which is the minimal hopcount, follows from the results of \cite{EskerHofstadHooghiemstra}. Our findings do not provide any information about paths with small hopcounts since we study paths with a small total weight and only take the hopcounts of such paths into account.

In this paper we assume for the typical weight $X$ that
\begin{equation}\label{eqn:probability_weight_zero}
\lambda\mathbb{P}(X=0)<1.
\end{equation}
This condition ensures that there are not too many edges with weight zero. Since the probability that an edge is present in $G(n,\lambda/n)$ and has weight zero is given by $\lambda\mathbb{P}(X=0)/n$, assumption \eqref{eqn:probability_weight_zero} implies that the subgraph of $G(n,\lambda/n)$ of edges with weight zero is a subcritical Erd\H{o}s--R\'{e}nyi graph.

Let also
$$
R(s) = \E\left[e^{-s X}\right],
$$
which is clearly defined for all $s \ge 0$, strictly decreasing, and such that $R(0) = 1$ and $R(\infty)=\mathbb{P}(X=0)$. Therefore, as $\lambda>1$ and using \eqref{eqn:probability_weight_zero}, there exists a unique $\alpha > 0$ that solves
\begin{equation}\label{eq:def_alpha} 
\lambda R(\alpha) = 1.
\end{equation}
With this $\alpha$, we can define a distribution $\tilde{F}$ via
\begin{equation}\label{eqn:F_tilde}
d\tilde{F}(x) = \lambda e^{-\alpha x} dF(x).
\end{equation}
Note here that $\tilde{F}$ is guaranteed to have finite exponential moments for parameters not exceeding $\alpha$ and hence, in particular, has all power moments. We will also use $\tilde{X}$ to denote a generic random variable with distribution $\tilde{F}$. Define further
\begin{equation} \label{eq:def_beta_gamma}
\beta = \frac{\bar{\sigma}^2}{\alpha \bar{\nu}^3} \quad \text{and} \quad \gamma = \frac{1}{\alpha \bar{\nu}},
\end{equation}
with $\bar{\nu} = \E[\tilde{X}]$ and $\bar{\sigma}^2 = \operatorname{Var}(\tilde{X})$.

We will obtain different results depending on whether the distribution of $X$ is non-arithmetic or arithmetic. Recall that the random variable $X$ has an arithmetic distribution with span $M$ if $\mathbb{P}(X\in M\mathbb{Z})=1$, where $M\mathbb{Z}=\{Mz:z\in\mathbb{Z}\}$, and $M$ is the largest positive real number $l>0$ such that $\mathbb{P}(X\in l\mathbb{Z})=1$. In order to study both cases simultaneously, we let 
\begin{equation}\label{eq:def_rho}
\varrho_n= \begin{cases} \frac{1}{\alpha} \log(n), & \quad \text{$X$ is non-arithmetic},\\ M\lfloor \log(n)/(M\alpha)\rfloor, & \quad \text{$X$ is arithmetic with span $M$}, \end{cases}
\end{equation}
for $n\in\mathbb{N}$. Here and throughout the rest of the paper we use the notation $\lfloor y\rfloor = \max\{z\in\mathbb{Z}: z\leq y\}$ for $y\in\mathbb{R}$.

While one can think of a point process as an at most countable collection of random points, formally we mean by a point process a random locally finite counting measure on $\mathbb{R}^d$ ($\mathbb{R}^d$ with $d\in\{1,2\}$ as an underlying space is sufficient for this paper). For $y\in\R^d$ we denote by $\delta_y$ the Dirac measure concentrated at $y$. In the following we consider the point processes
\begin{equation}\label{eq:def_psi_n}
\Psi_n = \sum_{p \in \mathcal{P}_1^n} \delta_{\left(\frac{H(p) - \gamma \log(n)}{\sqrt{\beta \log(n)}}, L(p) - \varrho_n \right)}
\end{equation}
for $n\in\mathbb{N}$ with $n\geq\lambda$. This means that $\Psi_n$ is the collection of all pairs of suitably transformed hopcounts and total weights of paths from vertex $1$ to vertex $n$ in $G(n,\lambda/n)$.

In our first main result we will show the convergence in distribution of $(\Psi_n)_{n\in\mathbb{N}}$ to a Cox process, which we introduce in the following. For $a\geq 0$ we denote by $\operatorname{Poisson}(a)$ a Poisson distribution with parameter $a$. We use the convention that $\operatorname{Poisson}(\infty)$ is the distribution concentrated at $\infty$ and write $\sim$ to indicate that a random object follows a distribution. We say that a point process $\Xi$ on $\mathbb{R}^d$ is a Poisson process with a locally finite intensity measure $\Lambda$ if $\Xi(A)\sim\operatorname{Poisson}(\Lambda(A))$ for all $A\in\mathcal{B}(\mathbb{R}^d)$ and $\Xi(A_1),\hdots,\Xi(A_k)$ are independent for all pairwise disjoint $A_1,\hdots,A_k\in\mathcal{B}(\mathbb{R}^d)$ with $k\in\mathbb{N}$. In this case we write $\Xi\sim\operatorname{Poisson}(\Lambda)$. One can think of Cox processes as a generalisation of Poisson processes, where the intensity measures are random. Let $\Theta$ be a random locally finite measure on $\mathbb{R}^d$. A point process $\Psi$ on $\mathbb{R}^d$ is called a Cox process directed by $\Theta$ if the conditional distribution of $\Psi$ given $\Theta$ is almost surely a Poisson process with intensity measure $\Theta$. Thus, for a Cox process $\Psi$ directed by the random measure $\Theta$ we use the notation $\Psi\sim\operatorname{Poisson}(\Theta)$.  

In order to provide the underlying random measure of the limiting Cox process of $\Psi_n$ as $n\to\infty$, we need a particular distribution. We denote by $W$ a random variable with $\mathbb{E}[W]=1$ satisfying the distributional fixed point equation
\begin{equation}\label{eqn:definition_W}
W\overset{d}{=} \sum_{i=1}^D e^{-\alpha X_i} W_i,
\end{equation}
where $\overset{d}{=}$ stands for equality in distribution, $(W_i)_{i\in\mathbb{N}}$ are i.i.d.\ copies of $W$, $D\sim\operatorname{Poisson}(\lambda)$, and $(X_i)_{i\in\mathbb{N}}$ are i.i.d.\ copies of $X$ such that $(W_i)_{i\in\mathbb{N}}$, $D$ and $(X_i)_{i\in\mathbb{N}}$ are independent. Note that the distribution of $W$ is unique and that $W$ is non-negative (see Lemma \ref{lem:exp_moments_W} and \cite{Roesler1992}).

We denote by $\overset{d}{\longrightarrow}$ convergence in distribution. For more details on convergence in distribution of point processes with respect to the vague topology we refer the reader to e.g.\ Chapter 23 of \cite{k21}. We are now able to state our first main result.

\begin{theo}\label{thm:main}
Assume that \eqref{eqn:probability_weight_zero} is satisfied. Let  $\mathbb{P}_{N}$ be the distribution of a standard Gaussian random variable $N$, and let $W_1$ and $W_2$ be independent copies of $W$.
\begin{itemize}
\item [a)] Let $X$ be non-arithmetic. Then,
$$
\Psi_n \overset{d}{\longrightarrow} \operatorname{Poisson}(W_1 W_2\mathbb{P}_{N}\otimes K_\mathbb{R}) \quad \text{as} \quad n\to\infty,
$$
where $K_{\mathbb{R}}$ is the measure on $\mathbb{R}$ with the density $u\mapsto \frac{1}{\bar{\nu}} e^{\alpha u}$.
\item [b)] Assume that $X$ is arithmetic with span $M$ and let $(n_m)_{m\in\mathbb{N}}$ be a strictly increasing sequence in $\mathbb{N}$ such that $\lim_{m\to\infty} M\lfloor \log(n_m)/(M\alpha) \rfloor - \log(n_m)/\alpha =\vartheta$. Then,
$$
\Psi_{n_m} \overset{d}{\longrightarrow} \operatorname{Poisson}(W_1 W_2e^{\alpha \vartheta} \mathbb{P}_{N}\otimes K_{M\mathbb{Z}}) \quad \text{as} \quad m\to\infty
$$
with
$$
K_{M\mathbb{Z}} = \frac{M}{\bar{\nu}} \sum_{j\in\mathbb{Z}} e^{\alpha j M} \delta_{jM}.
$$
\end{itemize}
\end{theo}

Note that the findings for the non-arithmetic and the arithmetic cases in Theorem \ref{thm:main} are consistent since $K_{M\mathbb{Z}}\overset{v}{\longrightarrow}K_{\mathbb{R}}$ and $\vartheta\to 0$ as $M\to 0$, where $\overset{v}{\longrightarrow}$ stands for vague convergence.

We denote by $\varnothing$ the null measure on $\mathbb{R}^2$, which corresponds to the empty point configuration. One has $\Psi_n\neq \varnothing$ if and only if $1$ and $n$ are connected in $G(n,\lambda/n)$ (i.e., there exists at least one path from $1$ to $n$). In order to derive the limit of the probability of this event, recall that the size of the largest component $\mathscr{C}_{\max}$ of $G(n,\lambda/n)$ behaves as $(1-\eta_\lambda) n$ as $n\to\infty$, where $\eta_\lambda$ is the extinction probability of a discrete-time Poisson branching process with mean offspring $\lambda$, and that all other components are much smaller (see, e.g., \cite[Theorem 4.8]{hof16}). Thus, we have
$$
\lim_{n\to\infty}\mathbb{P}(\Psi_n\neq\varnothing) = \lim_{n\to\infty}\mathbb{P}(1,n\in\mathscr{C}_{\max}) = (1-\eta_{\lambda})^2.
$$
Note that $\eta_\lambda$ is the unique solution in $[0,1)$ of the fixed point equation $s=G_\lambda(s)$, where $G_\lambda$ is the probability generating function of a $\operatorname{Poisson}(\lambda)$-distributed random variable. Since $\mathbb{P}(W=0)\in[0,1)$ satisfies $\mathbb{P}(W=0)=G_\lambda(\mathbb{P}(W=0))$ by \eqref{eqn:definition_W}, we have $\mathbb{P}(W=0)=\eta_\lambda$, whence
\begin{equation}\label{eqn:probability_empty_point_process}
\lim_{n\to\infty}\mathbb{P}(\Psi_n\neq\varnothing) = \mathbb{P}(W>0)^2 = \mathbb{P}(W_1,W_2>0).
\end{equation}
Moreover, the limiting Cox processes in Theorem \ref{thm:main} are different from $\varnothing$ if and only if $W_1,W_2>0$. Thus, the statements of Theorem \ref{thm:main} remain valid if we condition on the event that $1$ and $n$ are connected on the left-hand sides and on $W_1,W_2>0$ on the right-hand sides. 

Theorem \ref{thm:main} allows us to study the minimum total weight of a path from $1$ to $n$, which we denote by $L_{\min}$. As the path with the smallest total weight does not need to be unique, we let $P_{\min}$ be the number of paths with total weight $L_{\min}$ and let $H_{\min}^{(1)},\hdots,H_{\min}^{(P_{\min})}$ be their hopcounts ordered according to size, with $H^{(1)}_{\min}$ being the smallest one. We also define $H_{\min}=H_{\min}^{(1)}$. We note that $P_{\min}=0$ if and only if $1$ and $n$ are not connected and use the convention $H_{\min}=L_{\min}=\infty$ in this case.

\begin{theo}\label{thm:shortest_path}
Assume that \eqref{eqn:probability_weight_zero} is satisfied, and let $W_1$ and $W_2$ be independent copies of $W$.
\begin{itemize}
\item [a)] Let $X$ be non-arithmetic. Then, $\lim_{n\to\infty} \mathbb{P}(P_{\min}\in \{0,1\})=1$ and, for $z,u\in\mathbb{R}$,
\begin{multline*}
\lim_{n\to\infty}\mathbb{P}\left( \frac{H_{\min}-\gamma\log(n)}{\sqrt{\beta\log(n)}} \leq z, L_{\min} \leq \frac{1}{\alpha} \log(n) + u \right) \\
= \Phi(z) \left( 1 - \mathbb{E}\left[ \exp\left(-\frac{W_1W_2 e^{\alpha u}}{\alpha\bar{\nu}}\right) \right] \right),
\end{multline*}
where $\Phi$ is the distribution function of the standard Gaussian distribution.
\item [b)] Assume that $X$ is arithmetic with span $M$ and let $(n_m)_{m\in\mathbb{N}}$ be a strictly increasing sequence in $\mathbb{N}$ such that $\lim_{m\to\infty} M\lfloor \log(n_m)/(M \alpha) \rfloor - \log(n_m)/\alpha =\vartheta$. Then, for $u\in M\mathbb{Z}$, $k\in\mathbb{N}$ and $z_1,\hdots,z_k\in\mathbb{R}$,
\begin{align*}
& \lim_{m\to\infty} \mathbb{P}\left( L_{\min} =M\left\lfloor\frac{\log(n_m)}{M\alpha} \right\rfloor + u, P_{\min}=k, \frac{H^{(i)}_{\min}-\gamma\log(n_m)}{\sqrt{\beta\log(n_m)}} \leq z_i \text{ for all } i\in[k] \right) \\
& = \mathbb{E}\left[ \left( 1 - \exp\left(-\frac{W^\prime e^{\alpha (\vartheta+u)}}{\bar{\nu}(e^{\alpha M}-1)}\right)\right) \frac{ (W^\prime e^{\alpha (\vartheta+u)})^k }{\bar{\nu}^k} \exp\left(-\frac{W^\prime e^{\alpha(\vartheta+u)}}{\bar{\nu}}\right) \right] \\
& \quad \times  \mathbb{P}(N_{(1)}\leq z_1,\hdots,N_{(k)}\leq z_k),
\end{align*}
where $W^\prime=MW_1W_2$ and $N_{(1)}\leq \hdots \leq N_{(k)}$ are $k$ independent standard Gaussian random variables ordered by size.
\end{itemize}
\end{theo}

Again the statements of Theorem \ref{thm:shortest_path} remain valid if we condition on $H_{\min},L_{\min}<\infty$ on the left-hand side and on $W_1,W_2>0$ on the right-hand side. In this case $L_{\min}$ and $H_{\min}$ become asymptotically independent, as we would expect given other results for first passage percolation on random graphs; see, for example, \cite{bghk15,bhh14,bhh17}.

First passage percolation on the Erd\H{o}s-R\'enyi graph was considered in \cite{bhh11} for the case that $X$ follows an exponential distribution with parameter one. We note that this was later generalised to inhomogeneous random graphs in \cite{kk15}. In \cite{bhh11} the joint convergence of $(H_{\min},L_{\min})$ after a suitable transformation is shown, conditional on $1$ and $n$ being connected. This result can be also obtained from the conditional version of Theorem \ref{thm:shortest_path} a) discussed above. Our normalisations of the random variables and the limiting distribution of $H_{\min}$ are exactly as in \cite{bhh11}, while we obtain a different limiting distribution for $L_{\min}$. However, in their subsequent paper \cite{bhh17} the authors of \cite{bhh11} mention an error in their previous result. In \cite{bhh17} first passage percolation on the configuration model is studied and the weights are required to be non-negative and to have a density. For the configuration model with degrees from a Poisson distribution with parameter $\lambda$, the results from \cite{bhh17} provide the same limiting distributions for the whole point process $\Psi_n$ and $(H_{\min},L_{\min})$ as in our Theorems \ref{thm:main} a) and \ref{thm:shortest_path} a) (with the difference that it is conditioned on $1$ and $n$ being connected and on $W_1,W_2>0$). This is no surprise, as one can expect that this particular configuration model behaves as the Erd\H{o}s--R\'{e}nyi graph $G(n,\lambda/n)$, and means that our findings are consistent with those of \cite{bhh17}. In fact the approach of \cite{bhh17} was extended in \cite{bhh14} to uniform and rank-one inhomogeneous random graphs, which include Erd\H{o}s--R\'{e}nyi graphs as a special case. So \cite[Theorem 3.3]{bhh14} leads to Theorem \ref{thm:shortest_path} a) (conditioned on $1$ and $n$ being connected and $W_1,W_2>0$) under the assumption that the distribution of the weights has a density.

In order to better understand our limiting distributions and to make the comparison with the mentioned results from \cite{bhh11,bhh14,bhh17} easier, let us discuss the limiting point processes and the distribution of $W$ in more detail. The Cox process in Theorem \ref{thm:main} a) can be also written as
$$
\mathbf{1}\{W_1,W_2>0\} \sum_{i=1}^{\infty} \delta_{(Z_i,Y_i - \frac{1}{\alpha} \log(W_1) - \frac{1}{\alpha} \log(W_2))},
$$
where $(Z_i)_{i\in\mathbb{N}}$ are independent standard Gaussian random variables, $(Y_i)_{i\in\mathbb{N}}$ are the points of a Poisson process on $\mathbb{R}$ with intensity measure $K_{\mathbb{R}}$, and $(Z_i)_{i\in\N}$, $(Y_i)_{i\in\N}$, $W_1$ and $W_2$ are assumed to be independent. In this representation, the second coordinates are the points of a Poisson process shifted by the random quantity $- \frac{1}{\alpha} \log(W_1) - \frac{1}{\alpha} \log(W_2)$. For the situation of weights with an arithmetic distribution considered in Theorem \ref{thm:main} b) such a representation with a random shift is not available since the total weights are concentrated on $M\mathbb{Z}$.

So far we have defined $W$ via the distributional fixed point equation \eqref{eqn:definition_W}. In the following we provide two alternative characterisations of the distribution of $W$ in terms of branching processes. For this we assume that $\mathbb{P}(X=0)=0$. Consider a continuous-time Galton--Watson process starting from one individual and in which each individual has an independent lifetime with distribution $F$ and gives birth to a $\operatorname{Poisson}(\lambda)$-distributed number of offspring at the end of that lifetime. Let $N(t)$ be the number of individuals alive in this process at time $t$. It is known from \cite{Athreya1969} that there exists a random variable $\tilde{V}$ such that $N(t)/\mathbb{E}[N(t)]\overset{d}{\longrightarrow} \tilde{V}$ as $t\to\infty$ and that $\tilde{V}$ satisfies $\mathbb{E}[\tilde{V}]=1$ and
$$
\tilde{V} \overset{d}{=} e^{-\alpha X} \sum_{i=1}^D \tilde{V}_i,
$$
where $(\tilde{V}_i)_{i\in\mathbb{N}}$ are independent copies of $\tilde{V}$, $D\sim\operatorname{Poisson}(\lambda)$, and $D$, $X$, and $(\tilde{V}_i)_{i\in\mathbb{N}}$ are independent. We then let
$$
\tilde{W} = \sum_{i=1}^D \tilde{V}_i.
$$
For copies $(\tilde{W}_i)_{i\in\mathbb{N}}$ of $\tilde{W}$, copies $(X_i)_{i\in\mathbb{N}}$ of $X$, $\operatorname{Poisson}(\lambda)$-distributed random variables $D, (D_i)_{i\in\mathbb{N}}$ and copies $(\tilde{V}_i)_{i\in\mathbb{N}}$ and $(\tilde{V}_{i,j})_{i,j\in\mathbb{N}}$ of $\tilde{V}$, which are all independent, we obtain
$$
\sum_{i=1}^{D} e^{-\alpha X_i} \tilde{W}_i \overset{d}{=} \sum_{i=1}^{D} e^{-\alpha X_i} \sum_{j=1}^{D_i} \tilde{V}_{i,j} \overset{d}{=} \sum_{i=1}^{D} \tilde{V}_{i} \overset{d}{=} \tilde{W},
$$
whence $\tilde{W}$ is a solution of \eqref{eqn:definition_W}. Moreover, $W=\tilde{W}/\mathbb{E}[\tilde{W}]=\tilde{W}/\lambda$, where we used $\mathbb{E}[\tilde{V}]=1$, is a solution of the fixed point equation \eqref{eqn:definition_W} with expectation one. 

We remark that $\tilde{W}$ may be described as the limit of a slightly different branching process which is more common in the literature on first passage percolation. We assume that the first individual has a lifetime $0$ while all other individuals have independent lifetimes drawn from distribution $F$. Then, the number $\tilde{N}(t)$ of individuals alive at time $t$ has the same distribution as $\sum_{i=1}^D N_i(t)$ with independent $D\sim\operatorname{Poisson}(\lambda)$ and copies $(N_i(t))_{i\in\mathbb{N}}$ of $N(t)$. This implies that $\tilde{N}(t)/\mathbb{E}[N(t)]\overset{d}{\longrightarrow} \tilde{W}$ as $t\to\infty$. The unusual assumption that the first individual has a lifetime $0$ is due to the fact that the branching process corresponds to the process of exploration, starting at time $0$, of the neighbourhood of a vertex. The vertex itself is explored immediately (corresponding to the $0$ lifetime), while its immediate neighbours (corresponding to children) are explored after i.i.d.\ times. Individuals from future generations are also explored after i.i.d.\ times. The terms `active' and `period of activity' may be better to describe individuals in this process and their lifetimes, respectively.

Point process convergence of extremal processes to Cox processes was obtained for other models of first passage percolation. For directed first passage percolation on the hypercube with exponential weights, point process convergence of the total weights to a Cox process was established in \cite{KSS}. Convergence in distribution of the smallest total weight to the first point of a Cox process was shown for the complete graph in \cite{Bhamidi08,bh12,eghn20,Janson99} and for dense {E}rd\H{o}s-{R}\'{e}nyi graphs in \cite{Bhamidi08}. The random intensity measure of the Cox process in \cite{bh12} is even driven by the random variables $W_1$ and $W_2$ as in Theorem \ref{thm:main}.

Our proof of Theorem \ref{thm:main} relies on the method of moments, employing combinatorial arguments and results from \cite{im10_jap} to control sums of weights. This is a different approach than in the existing literature and, in particular, in \cite{bhh11,bhh14,bhh17,kk15}, where branching processes play key roles. The main achievement of our work is that we can deal with very general weight distributions. The proof strategies of \cite{bhh11} and \cite{kk15} heavily rely on the Markovian structure due to the exponentially distributed weights and, thus, cannot be extended to more general weight distributions. The approach in \cite{bhh17} and \cite{bhh14} only requires the weights to have a density, but is still more restrictive than what is assumed throughout this paper.

It is a natural question whether our approach can also be applied to the configuration model and inhomogeneous random graphs. We conjecture that our strategy can be adapted to derive results as for the configuration model in \cite{bhh17} and for inhomogeneous random graphs in \cite{bhh14,kk15} under weaker assumptions on the weights, and analogous results for non-arithmetic weight distributions. However, as we use the method of moments, this requires the convergence of all moments of the degree sequence or of the degree of a typical vertex, which are much stronger assumptions than those made in \cite{bhh14}, \cite{bhh17} and \cite{kk15}, respectively. Whether one can nevertheless generalise the findings of \cite{bhh14}, \cite{bhh17} and \cite{kk15} to more general weight distributions without changing the moment assumptions is a question for further research. 

All our results are subject to the assumption \eqref{eqn:probability_weight_zero}, which ensures that there are not too many edges with weight zero. It is natural to wonder what happens if
$$
\lambda\mathbb{P}(X=0)=1 \quad \text{or} \quad \lambda\mathbb{P}(X=0)>1.
$$
In these cases the edges with weight zero form a critical or supercritical Erd\H{o}s-R\'enyi graph. One can expect that, at least in the latter case, completely different behaviour will be observed as a huge number of edges with weight zero can be used to travel between the vertices $1$ and $n$.

\section{Proofs of Theorems \ref{thm:main} and \ref{thm:shortest_path}}\label{sec:proof}

In this section we give the proofs of Theorems \ref{thm:main} and \ref{thm:shortest_path}. We start with a short outline of our proof strategy for Theorem \ref{thm:main}. Although we deal with the non-arithmetic and arithmetic cases simultaneously, we focus in this description on the non-arithmetic situation to keep the notation simple. Our aim is to show for a suitable class of sets $A$ that
\begin{equation}\label{eqn:strategy_proof}
\Psi_n(A)\overset{d}{\longrightarrow} \Psi_{\mathbb{R}^2}(A) \quad \text{as} \quad n\to\infty,
\end{equation}
where $\Psi_{\mathbb{R}^2}$ is the limiting Cox process in the non-arithmetic case. In other words, we establish that the number of points in $A$ of the considered point processes converges in distribution to the number of points of the limiting point process in $A$. Note that \eqref{eqn:strategy_proof} for a suitable class of sets $A$ is sufficient for the claimed point process convergence in Theorem \ref{thm:main}.

To verify \eqref{eqn:strategy_proof}, we employ the method of moments. First we compute the factorial moments of the limiting point processes in Section \ref{sec:moments_limit}. As they involve moments of the solution $W$ of the distributional fixed point equation \eqref{eqn:definition_W}, we study the existence as well as properties of $W$ and derive in particular a recursive formula for its moments. When counting paths with small sums of weights, we need to deal with certain sums of probabilities. We prepare for this with estimates in Section \ref{sec:large_deviations}, which heavily rely on results from \cite{im10_jap} to control the sums of weights. Thereafter, we compute the first moment of the prelimit point process $\Psi_n$. Apart from that, we show in Section \ref{sec:first_moments} that $\Psi_n$ can be approximated by a point process $\Psi_n^{(a,w)}$, defined precisely in \eqref{eq:def_Psi_n_a,w} below, that takes into account only short paths that are not crossed by other short paths. In fact, we will first establish \eqref{eqn:strategy_proof} for $\Psi_n^{(a,w)}$ instead of $\Psi_n$, which then carries over to $\Psi_n$. To compute the factorial moments of $\Psi_n^{(a,w)}$ in Section \ref{sec:higher_moments}, we introduce some weighted sums $M^{(r)}$, $r\in\mathbb{N}$, over trees in Section \ref{sec:trees}, which satisfy the same recursion as the moments of $W$. Thanks to this joint recursive structure one can conclude that the factorial moments of $\Psi_{n}^{(w,a)}$ converge to those of the limiting Cox processes. This is done the final step of the proof in Section \ref{sec:proof_completion}.

Finally, Theorem \ref{thm:shortest_path} is derived from Theorem \ref{thm:main} in Section \ref{sec:proof_shortest_path}.


\subsection{Factorial moments of the limiting point processes}\label{sec:moments_limit}

First we study the existence and uniqueness of $W$ as well as its non-negativity and the existence of moments.

\begin{lemma} \label{lem:exp_moments_W}
The distributional fixed point equation \eqref{eqn:definition_W} has a unique solution $W$ with $\mathbb{E}[W]=1$. Moreover, $\mathbb{P}(W\geq 0)=1$ and there exists $t>0$ such that
$$
\E\left[e^{tW}\right] < \infty.
$$
\end{lemma}

\begin{proof}
Let us introduce
$$
W(1)= 1 \quad \text{and} \quad W(n+1) = \sum_{i=1}^D e^{-\alpha X_i} W_i(n) \quad \text{for} \quad  n\in\mathbb{N},
$$
where $(X_i)_{i\in\mathbb{N}}$ all have the same distribution as $X$, $(W_i(n))_{i\in\mathbb{N}}$ all have the same distribution as $W(n)$, $D\sim\operatorname{Poisson}(\lambda)$ and all these random variables are independent. Defining
$$
T_i = e^{-\alpha X_i} \mathbf{1}\{D \ge i\}
$$
for $i\in\mathbb{N}$, we have that
$$
\E\left[\sum_{i=1}^\infty T_i^2\right] = \E\left[\sum_{i=1}^D e^{-2 \alpha X_i}\right] = \lambda \E\left[e^{-2 \alpha X}\right] < 1
$$
due to the definition of $\alpha$ and \eqref{eqn:F_tilde}. Thus, it follows from \cite{Roesler1992} (see Theorem 4 and the remark after it) that the sequence $(W(n))_{n\in\mathbb{N}}$ converges in distribution to a unique fixed point which satisfies \eqref{eqn:definition_W} and has mean $1$. This provides the existence and uniqueness of $W$. Since $(W(n))_{n\in\mathbb{N}}$ are non-negative by construction, $W$ is also non-negative as their limit.

We now show by induction that there exist $c > 0$ and $t > 0$ such that
\begin{equation} \label{eq:ind_exp_moments}
\E\left[e^{sW(n)}\right]  \le e^{s+cs^2}\quad \text{for all} \quad 0 \le s \le t \quad\text{and} \quad n\in\mathbb{N}.
\end{equation}
For $n=1$, the above holds. Assume now that it is valid for an arbitrary $n$. As $e^{-\alpha X} \le 1$ a.s., this implies that
$$
\E\left[e^{se^{-\alpha X} W(n)}|X\right] \le e^{se^{-\alpha X} + c s^2e^{-2\alpha X}} \quad \text{a.s.}
$$
Hence,
\begin{align*}
\E\left[e^{sW(n+1)}\right] & = \exp\left\{\lambda\left(\E\left[e^{se^{-\alpha X} W(n)}\right] -1 \right) \right\} 
\\ & \le \exp\left\{\lambda\left(\E\left[e^{se^{-\alpha X} + c s^2e^{-2\alpha X}} \right] - 1\right)\right\}
\\ & \le \exp\big\{\lambda s \E\left[e^{-\alpha X}\right] + \lambda (c+1)s^2 \E\left[e^{-2\alpha X}\right] \\
& \quad  + 2 \lambda c s^3\E\left[e^{-3\alpha X}\right] + \lambda s^4 c^2\E\left[e^{-4\alpha X}\right] \big\},
\end{align*}
as long as $t$ is chosen small enough so that
\begin{equation}\label{eqn:condition_t}
e^u \le 1 + u +u^2 \quad \text{for all} \quad 0\le u\leq t+ct^2.
\end{equation}
As $\lambda \E\left[e^{-\alpha X}\right] =1$, it is sufficient to choose $c$ so that
$$
c > \frac{\lambda \E\left[e^{-2\alpha X}\right]}{1 - \lambda \E\left[e^{-2\alpha X}\right]}
$$
and $t$ small enough so that \eqref{eqn:condition_t} is satisfied and
$$
\left(\lambda (c+1) \E\left[e^{-2\alpha X}\right] - c \right)  + 2 \lambda c t\E\left[e^{-3\alpha X}\right] + \lambda t^2 c^2\E\left[e^{-4\alpha X}\right] \le 0
$$
to demonstrate that \eqref{eq:ind_exp_moments} holds for all $n$.

As $W(n) \overset{d}{\longrightarrow} W$ as $n\to\infty$, thanks to \cite[Theorem 5.31]{k21} there exists a probability space and random variables $(V(n))_{n\in \mathbb{N}}$ and $V$ defined on it such that $V(n) \overset{d}{=}W(n) $ for $n\in\mathbb{N}$, $V \overset{d}{=} W$ and $V(n) \overset{a.s.}{\longrightarrow} V$ as $n\to\infty$. This, along with \eqref{eq:ind_exp_moments} and the Fatou lemma, implies the statement of the lemma.
\end{proof}

For $u\in\R$ and $r\in\mathbb{N}$ we define $(u)_r= u(u-1)\cdots (u-r+1)$ and recall that for a random variable $Y$ the expectation $\E[(Y)_r]$ is called the $r$-th factorial moment of $Y$ if it is well defined. Let $\Pi(r)$ denote the set of partitions of $[r]$, which are collections of pairwise disjoint non-empty subsets of $[r]$ whose union is $[r]$. By $|I|$ we denote the cardinality of a finite set $I$.

Recall again that $\E[W]=1$ by definition. The following lemma provides a relation for higher moments.

\begin{lemma} \label{lem:recursion_moments_W}
The following recursive relation holds for $W$ as in Lemma \ref{lem:exp_moments_W} and all integers $r \ge 2$:
\begin{equation*} 
\E\left[W^r\right] = \frac{1}{1-\lambda \E\left[e^{-r \alpha X}\right]}\sum\limits_{\pi \in \Pi(r) \setminus \{[r]\}} \prod\limits_{I \in \pi} \lambda \E\left[e^{-|I| \alpha X}\right] \E\left[W^{|I|}\right].
\end{equation*}
\end{lemma}

\begin{proof}
This follows from the fact that, since $D$ is a Poisson$(\lambda)$ random variable, its $r$-th factorial moment is equal to $\lambda^r$. Hence, using \eqref{eqn:definition_W},
\begin{align*}
\E\left[W^r\right] & = \sum\limits_{\pi \in \Pi(r)} \prod\limits_{I \in \pi} \lambda \E\left[e^{-|I| \alpha X}\right] \E\left[W^{|I|}\right] 
\\ & = \lambda \E\left[e^{-r\alpha X}\right] \E\left[W^r\right] + \sum\limits_{\pi \in \Pi(r) \setminus \{[r]\}} \prod\limits_{I \in \pi} \lambda \E\left[e^{-|I| \alpha X}\right] \E\left[W^{|I|}\right],
\end{align*}
as required.
\end{proof}

Next we compute the factorial moments of the number of points in a given set for a Cox process as in Theorem \ref{thm:main}.

\begin{lemma} \label{lem:factorial_moments_limit}
For a locally finite measure $\Lambda$ on $\mathbb{R}^2$ and independent copies $W_1$ and $W_2$ of $W$, let $\Psi\sim\operatorname{Poisson}(W_1W_2\Lambda)$. Then, for $A \in \mathcal{B}(\R^2)$ such that $\Lambda(A) < \infty$ and $r\in\mathbb{N}$,
\begin{equation*} 
\E\left[\left(\Psi(A) \right)_r\right] = \Lambda(A)^r \left(\E\left[W^r\right]\right)^2.
\end{equation*}
\end{lemma}

\begin{proof}
Since $\Psi(A)$ conditioned on $W_1$ and $W_2$ follows a Poisson distribution whose parameter is $W_1W_2\Lambda(A)$ and the $r$-th factorial moment of a Poisson distribution with parameter $a$ is $a^r$, we obtain
\begin{align*}
\E[(\Psi(A))_r] &= \E[\E[(\Psi(A))_r|W_1,W_2] ] = \E[ (W_1W_2\Lambda(A))^r ]\\
& = \E[ W_1^r ] \E[ W_2^r ] \Lambda(A)^r = \Lambda(A)^r (\E[ W^r ])^2,
\end{align*}
which is the desired identity.
\end{proof}

Since we will use the method of moments to show convergence in distribution of $(\Psi_n(A))_{n\in\mathbb{N}}$ to $\Psi(A)$ for $\Psi$ as in the previous lemma and suitable sets $A$, we need that the distribution of $\Psi(A)$ is characterised by its moments, which we establish in the following lemma.

\begin{lemma} \label{lem:moments_define_distribution}
For $\Psi$ and $A\in\mathcal{B}(\R^2)$ as in Lemma \ref{lem:factorial_moments_limit}, the distribution of $\Psi(A)$ is uniquely determined by its moments.
\end{lemma}

\begin{proof}
As $\Psi(A)$ is almost surely non-negative, a sufficient condition (referred to as Carleman's condition; see page 353 of \cite{s16}) for its distribution to be defined by its moments is
\begin{equation} \label{eq:stieltjes}
\sum_{r=1}^\infty \E\left[\Psi(A)^r\right]^{-\frac{1}{2r}} = \infty.
\end{equation}
Note now that
\begin{align*}
\E\left[\Psi(A)^r\right] & = \E\left[\Psi(A)^r \mathbf{1}\{\Psi(A) - r +1 \ge \Psi(A)/2\} \right] \\
& \quad + \E\left[\Psi(A)^r \mathbf{1}\{\Psi(A) - r +1 < \Psi(A)/2\} \right]
\\ & \le 2^r \E[(\Psi(A))_r] + 2^r r^r = 2^r \left(\Lambda(A)^r \left(\E\left[W^r\right]\right)^2 + r^r \right),
\end{align*}
where we used Lemma \ref{lem:factorial_moments_limit} in the last step.
Thanks to Lemma \ref{lem:exp_moments_W} and as $\E\left[e^{tW}\right] \ge t^r \E\left[W^r\right]/r!$, there exist constants $0< t, C < \infty$ such that $\E\left[W^r\right] \le C r!/t^r$ for all $r\ge 1$. This implies
$$
\E\left[\Psi(A)^r\right] \le 2^r \left(\frac{\Lambda(A)^r}{t^{2r}} C^2 (r!)^2 + r^r\right) = O\left(2^r \frac{\Lambda(A)^r}{t^{2r}} r^{2r+1} e^{-2r}\right),
$$
which proves \eqref{eq:stieltjes} and, thus, the lemma.
\end{proof}

\subsection{Truncated exponential renewal functions}\label{sec:large_deviations}

Throughout the proof of Theorem \ref{thm:main} we count short paths to obtain moment estimates for our prelimit processes. For this it is important to control products of probabilities that some path is present in the {E}rd\H{o}s--{R}\'{e}nyi graph and that the sum of its weights does not exceed a given threshold. In this subsection, we derive estimates for sums of such expressions that will be employed later on.

In order to study the non-arithmetic and the arithmetic cases simultaneously, we define
\begin{equation}\label{eq:def_tau_n}
\tau_n= n e^{-\alpha \varrho_n}
\end{equation}
for $n\in\mathbb{N}$ (noting that this simplifies to $\tau_n=1$ in the non-arithmetic case) and
\begin{equation}\label{eq:def_Lambda}
\Lambda=\begin{cases} \mathbb{P}_N\otimes K_\mathbb{R}, & \quad \text{$X$ is non-arithmetic},\\ \mathbb{P}_N\otimes K_{M\mathbb{Z}}, & \quad \text{$X$ is arithmetic with span $M$}. \end{cases}
\end{equation}
For $n\in\mathbb{N}$ and $z\in\R$ we define
\begin{equation}\label{eq:def_of_k_n(z)}
k_n(z)= \lfloor z \sqrt{\beta \log(n)} + \gamma \log(n)\rfloor
\end{equation}
with $\beta$ and $\gamma$ as in \eqref{eq:def_beta_gamma}. We let $S_k=X_1+\cdots+X_k$ for $k\in\mathbb{N}$, where $(X_i)_{i\in\mathbb{N}}$ are i.i.d.\ random variables with distribution $F$.

\begin{lemma} \label{lem:renewal}
\begin{itemize}
\item [a)] For $\lambda$ as in \eqref{eqn:probability_weight_zero} and $\alpha$ as in \eqref{eq:def_alpha}, there exists a constant $C\in(0,\infty)$ such that
$$
\sum_{\ell=1}^\infty \lambda^{\ell} \mathbb{P}(S_\ell\leq x) \leq C e^{\alpha x}
$$
for all $x\in[0,\infty)$.
\item [b)] For $\lambda$ as in \eqref{eqn:probability_weight_zero} and $u,z\in\mathbb{R}$,
$$
\lim_{n\to\infty} \frac{\tau_n}{n} \sum_{\ell=1}^{k_n(z)} \lambda^\ell \mathbb{P}\left(S_\ell \leq \varrho_n + u \right) = \Lambda((-\infty,z]\times(-\infty,u]),
$$
where $\varrho_n$ is defined by \eqref{eq:def_rho}.
\end{itemize}
\end{lemma}

\begin{proof}
For $x\geq 0$ we define
$$
V(x) = \sum_{\ell=0}^\infty \lambda^\ell \mathbb{P}(S_\ell\leq x).
$$
Note that
$$
R:= -\log \inf_{t\geq 0} \mathbb{E}[e^{-t X}] = \begin{cases} \infty, & \quad \mathbb{P}(X=0)=0,\\ -\log(\mathbb{P}(X=0)), & \quad \mathbb{P}(X=0)\in(0,1). \end{cases}
$$
Thanks to \eqref{eqn:probability_weight_zero}, we have that $\log(\lambda)<R$. Thus, it follows from Theorem 2.2 of \cite{im10_jap} (with $a=\log(\lambda)$ there) that
\begin{equation}\label{eqn:limit_V_non_arithmetic}
\lim_{x\to\infty} \frac{V(x)}{e^{\alpha x}} = \frac{1}{\alpha \bar{\nu}}
\end{equation}
in the non-arithmetic case and that
\begin{equation}\label{eqn:limit_V_arithmetic}
\lim_{k\to\infty} \frac{V(kM)}{e^{\alpha kM}} = \frac{M}{1-e^{-\alpha M}} \frac{1}{\bar{\nu}} 
\end{equation}
with the limit along the natural numbers in the arithmetic case.

As $e^{\alpha x} \geq 1$ for all $x\geq 0$ and $V$ is non-decreasing and, thus, bounded on every compact interval, \eqref{eqn:limit_V_non_arithmetic} and \eqref{eqn:limit_V_arithmetic} imply that there exists a constant $C\in(0,\infty)$ such that
\begin{equation}\label{eqn:bounded_V}
\frac{V(x)}{e^{\alpha x}} \leq C
\end{equation}
for all $x\geq 0$. This proves part a).

Next we show b) simultaneously for both the non-arithmetic and the arithmetic cases. To this end we define, for $n\in\mathbb{N}$,
$$
\varrho_n(u) = \frac{1}{\alpha} \log(n) + u
$$
for the non-arithmetic case and
$$
\varrho_n(u) = M \left\lfloor \frac{1}{M\alpha} \log(n) \right\rfloor + M \left\lfloor \frac{u}{M} \right\rfloor
$$
for the arithmetic case. Since $\varrho_n(u)$ is the smallest multiple of $M$ not exceeding $\varrho_n + u$ in the arithmetic case, we can replace the latter term inside the probabilities in b) by $\varrho_n(u)$.

If $X$ is non-arithmetic, we derive from the definitions of $\tau_n$, $V$ and $\varrho_n(u)$ as well as \eqref{eqn:limit_V_non_arithmetic} that
\begin{equation}\label{eqn:V_infinity_non_arithmetic}
\begin{split}
\lim_{n\to\infty} \frac{\tau_n}{n}  \sum_{\ell=1}^{\infty} \lambda^\ell \mathbb{P}\left(S_\ell \leq \varrho_n(u) \right) & = \lim_{n\to\infty} \frac{V\left(\varrho_n(u) \right)}{n} = \lim_{n\to\infty} \frac{e^{\alpha \varrho_n(u)}}{n} \frac{V\left(\varrho_n(u) \right)}{e^{\alpha \varrho_n(u)}} = \frac{e^{\alpha u}}{\alpha \bar{\nu}}.
\end{split}
\end{equation}
For the case that $X$ is arithmetic we have
\begin{equation}\label{eqn:limit_prefactor}
\frac{\tau_n e^{\alpha \varrho_n(u)}}{n} = \frac{\tau_n e^{\alpha \varrho_n + \alpha M \lfloor u/M\rfloor} }{n} = \frac{n e^{-\alpha \varrho_n} e^{\alpha \varrho_n + \alpha M \lfloor u/M\rfloor} }{n} = e^{\alpha M \lfloor u/M\rfloor},
\end{equation}
which, together with \eqref{eqn:limit_V_arithmetic}, leads to 
\begin{equation}\label{eqn:V_infinity_arithmetic}
\begin{split}
\lim_{n\to\infty} \frac{\tau_n}{n}  \sum_{\ell=1}^{\infty} \lambda^\ell \mathbb{P}\left(S_\ell \leq \varrho_n(u) \right) & = \lim_{n\to\infty} \frac{\tau_n V\left(\varrho_n(u) \right)}{n} = \lim_{n\to\infty} \frac{\tau_n e^{\alpha \varrho_n(u)}}{n} \frac{V\left(\varrho_n(u) \right)}{e^{\alpha \varrho_n(u)}} \\
& = \frac{M}{1-e^{-\alpha M}} \frac{e^{\alpha M\lfloor u/M\rfloor}}{\bar{\nu}}.
\end{split}
\end{equation}
In the last step we used that $\varrho_n(u)$ is a multiple of $M$, which allows us to apply \eqref{eqn:limit_V_arithmetic}.

Next we compute for
$$
U_n := \frac{\tau_n}{n}  \sum_{\ell=k_n(z)+1}^\infty \lambda^\ell \mathbb{P}\left(S_\ell \leq \varrho_n(u) \right)
$$
the limit of $(U_n)_{n\in\mathbb{N}}$. Let $(S_\ell')_{\ell\in\mathbb{N}}$ be independent copies of $(S_\ell)_{\ell\in\mathbb{N}}$. We have that
\begin{align*}
U_n & = \frac{\tau_n}{n}  \sum_{\ell=0}^\infty \lambda^{k_n(z)+1+\ell} \mathbb{P}\left(S_{k_n(z)+1+\ell} \leq \varrho_n(u) \right)\\
& = \frac{\tau_n}{n}  \sum_{\ell=0}^\infty \lambda^{k_n(z)+1+\ell} \mathbb{P}\left(S_{k_n(z)+1} + S_\ell' \leq \varrho_n(u) \right) \\
& = \frac{\tau_n}{n} \mathbb{E} \left[ \lambda^{k_n(z)+1} \mathbf{1}\left\{ S_{k_n(z)+1} \leq \varrho_n(u) \right\} \sum_{\ell=0}^\infty \lambda^\ell \mathbb{P}\left( S_\ell' \leq \varrho_n(u)  -S_{k_n(z)+1} \big| S_{k_n(z)+1} \right) \right] \\
& = \frac{\tau_n}{n} \mathbb{E} \left[ \lambda^{k_n(z)+1} \mathbf{1}\left\{ S_{k_n(z)+1} \leq \varrho_n(u) \right\} V\left( \varrho_n(u) -S_{k_n(z)+1}\right) \right] \\
& = \frac{\tau_n e^{\alpha \varrho_n(u)}}{n} \mathbb{E} \left[ \lambda^{k_n(z)+1} \mathbf{1}\left\{ S_{k_n(z)+1} \leq \varrho_n(u) \right\} e^{-\alpha S_{k_n(z)+1}} \frac{V\left(\varrho_n(u)  -S_{k_n(z)+1}\right)}{e^{\alpha \left( \varrho_n(u)  -S_{k_n(z)+1} \right)}} \right].
\end{align*}
Letting $\widetilde{S}_\ell := \sum_{i=1}^\ell \widetilde{X}_i$ for $\ell\in\mathbb{N}$ with i.i.d.\ random variables $(\widetilde{X}_i)_{i\in\mathbb{N}}$ distributed according to $\tilde{F}$ as defined in \eqref{eqn:F_tilde}, we obtain
\begin{equation}\label{eqn:representation_U_n}
U_n = \frac{\tau_n e^{\alpha \varrho_n(u)}}{n} \mathbb{E}\left[ \mathbf{1}\left\{ \widetilde{S}_{k_n(z)+1} \leq \varrho_n(u) \right\} \frac{V\left( \varrho_n(u)  -\widetilde{S}_{k_n(z)+1}\right)}{e^{\alpha \left( \varrho_n(u)  -\widetilde{S}_{k_n(z)+1} \right)}} \right].
\end{equation}
We denote the factors in the expectation by $I_n$ and $V_n$. Recall that $\bar{\nu}$ and $\bar{\sigma}^2$ are the expectation and the variance of a random variable distributed according to $\tilde{F}$. It follows from the central limit theorem that, for $y\in\{0,1\}$,
\begin{align*}
& \lim_{n\to\infty} \mathbb{P}\left( \widetilde{S}_{k_n(z)+1} \leq \varrho_n(u) - y ( \log(n) )^{1/4} \right) \\
& = \lim_{n\to\infty} \mathbb{P}\left( \frac{\widetilde{S}_{k_n(z)+1} - \bar{\nu} (k_n(z)+1)}{\bar{\sigma}\sqrt{k_n(z)+1} } \leq \frac{ \varrho_n(u) - y ( \log(n) )^{1/4} - \bar{\nu} (k_n(z)+1)}{\bar{\sigma}\sqrt{k_n(z)+1} } \right) \\
& = \Phi\left( \lim_{n\to\infty} \frac{ \varrho_n(u) - y ( \log(n) )^{1/4} - \bar{\nu} (k_n(z)+1)}{\bar{\sigma}\sqrt{k_n(z)+1} } \right),
\end{align*}
where $\Phi$ is the distribution function of the standard Gaussian distribution.
The definitions of $\varrho_n(u)$, $k_n(z)$, $\beta$ and $\gamma$ (see \eqref{eq:def_beta_gamma}) imply that
\begin{align*}
& \lim_{n\to\infty} \frac{ \varrho_n(u) - y ( \log(n) )^{1/4} - \bar{\nu} (k_n(z)+1)}{\bar{\sigma}\sqrt{k_n(z)+1} } \\
& = \lim_{n\to\infty} \frac{ \frac{1}{\alpha} \log(n) - y ( \log(n) )^{1/4} - \bar{\nu} (z\sqrt{\beta\log(n)}+\gamma\log(n))}{\bar{\sigma}\sqrt{z\sqrt{\beta\log(n)}+\gamma\log(n)+1} } \\
& = \lim_{n\to\infty} \frac{ \left(\frac{1}{\alpha} - \bar{\nu}\gamma \right) \log(n) - \bar{\nu} z\sqrt{\beta\log(n)}}{\bar{\sigma}\sqrt{\gamma\log(n)} } = -\frac{\bar{\nu}\sqrt{\beta}}{\bar{\sigma}\sqrt{\gamma}}z = -z,
\end{align*}
whence
$$
\lim_{n\to\infty} \mathbb{P}\left( \widetilde{S}_{k_n(z)+1} \leq \varrho_n(u) - y ( \log(n) )^{1/4} \right) = \Phi(-z).
$$
For $y=0$ this yields that $(I_n)_{n\in\mathbb{N}}$ converges in distribution to a random variable $I$ with $\mathbb{P}(I=1)=1-\mathbb{P}(I=0)=\Phi(-z)$. From the case $y=1$ we obtain that the random variables $(I_n')_{n\in\mathbb{N}}$ with
$$
I_n' = \mathbf{1} \left\{ \widetilde{S}_{k_n(z)+1} \leq \varrho_n(u) - ( \log(n) )^{1/4}  \right\}
$$
also converge in distribution to $I$. 

For any constant $c\in\mathbb{R}$ we have
$$
\mathbb{E}[ I_n V_n] = \mathbb{E}[ (I_n-I_n') V_n] + \mathbb{E}\left[ I_n' \left(V_n-c\right)\right] + \mathbb{E}\left[ I_n' c\right].
$$
From \eqref{eqn:bounded_V}, $I_n'\leq I_n$ and $V_n\geq 0$ for $n\in\mathbb{N}$, and the convergence in distribution of $(I_n)_{n\in\mathbb{N}}$ and $(I_n')_{n\in\mathbb{N}}$ to $I$, we derive
$$
0\leq \lim_{n\to\infty} \mathbb{E}[ (I_n-I_n') V_n] \leq C \lim_{n\to\infty} (\mathbb{E}[I_n] - \mathbb{E}[I_n']) = 0
$$
and
$$
\lim_{n\to\infty} \mathbb{E}\left[ I_n' c\right] = c \mathbb{E}[I] = c \Phi(-z),
$$
whereas
\begin{multline*}
\left| \mathbb{E}\left[ I_n' \left(V_n-c\right)\right] \right| \\
 \leq \mathbb{E}\left[ \mathbf{1}\left\{ (\log(n))^{1/4} \leq \varrho_n(u) - \widetilde{S}_{k_n(z)+1} \right\} \left| \frac{V\left( \varrho_n(u)  -\widetilde{S}_{k_n(z)+1}\right)}{e^{\alpha \left( \varrho_n(u)  -\widetilde{S}_{k_n(z)+1}\right)}} - c \right| \right].
\end{multline*}
The right-hand side is bounded by
$$
\sup_{t\geq (\log(n))^{1/4}} \left| \frac{V(t)}{e^{\alpha t}} - c \right| \quad \text{and} \quad \sup_{k\in\mathbb{N}_0: kM\geq (\log(n))^{1/4}} \left| \frac{V(kM)}{e^{\alpha kM}} - c \right|
$$
in the non-arithmetic and the arithmetic cases, respectively. For the choices $c = \frac{1}{\alpha \bar{\nu}}$ and $c=\frac{M}{(1-e^{-\alpha M})\bar{\nu}}$ these expressions vanish as $n\to\infty$ by \eqref{eqn:limit_V_non_arithmetic} and \eqref{eqn:limit_V_arithmetic}. Thus, we have shown that
$$
\lim_{n\to\infty} \mathbb{E}[ I_n V_n] = \frac{\Phi(-z)}{\alpha \bar{\nu}}
$$
in the non-arithmetic case and
$$
\lim_{n\to\infty} \mathbb{E}[ I_n V_n] = \frac{M}{1-e^{-\alpha M} } \frac{\Phi(-z)}{\bar{\nu}}
$$
in the arithmetic case. Together with $\frac{e^{\alpha \varrho_n(u)}}{n} =e^{\alpha u}$ for the non-arithmetic case and \eqref{eqn:limit_prefactor} for the arithmetic case, we obtain from \eqref{eqn:representation_U_n} that
$$
\lim_{n\to\infty} U_n = \frac{e^{\alpha u}}{\alpha \bar{\nu}}\Phi(-z)
$$
in the non-arithmetic case and that
$$
\lim_{n\to\infty} U_n =  \frac{M }{1-e^{-\alpha M} } \frac{e^{\alpha M\lfloor u/M\rfloor}}{\bar{\nu}}\Phi(-z)
$$
in the arithmetic case. Combining these identities with \eqref{eqn:V_infinity_non_arithmetic}, \eqref{eqn:V_infinity_arithmetic} and 
$$
\Lambda((-\infty,z]\times(-\infty,u]) = \begin{cases} \frac{e^{\alpha u}}{\alpha \bar{\nu}}\Phi(z), & \quad \text{$X$ is non-arithmetic},\\ \frac{M }{1-e^{-\alpha M} } \frac{e^{\alpha M\lfloor u/M\rfloor}}{\bar{\nu}}\Phi(z), & \quad \text{$X$ is arithmetic with span $M$}, \end{cases}
$$
completes the proof of part b).
\end{proof}

\subsection{Counting short paths: first moments}\label{sec:first_moments}

We now move on to the calculation of the moments of the prelimit point processes, starting with the first moment. For distinct vertices $u$ and $v$ in $G(n,\lambda/n)$ we denote the existence of an edge between $u$ and $v$ by $u\leftrightarrow v$ and by $X_{\{u,v\}}$ the weight of this edge. For any set $A$ and $k\in\mathbb{N}$, we write $A^k_{\not=}$ for the set of $k$-tuples of pairwise distinct elements of $A$. We define $\mathscr{P}_1^n=\{(v_0,\hdots,v_\ell): \ell\in\{1,\hdots,n-1\}, (v_0,\hdots,v_\ell)\in[n]^{\ell+1}_{\neq}, v_0=1, v_{\ell}=n \}$, which is the set of all potential paths from $1$ to $n$ in the complete graph with vertex set $[n]$. Recall that $\mathcal{P}_1^n$ is the set of all paths from $1$ to $n$ in $G(n,\lambda/n)$. Thus, for $p=(v_0,\hdots,v_\ell)\in\mathscr{P}_1^n$ we have $p\in\mathcal{P}_1^n$ if all edges $\{v_0,v_1\},\hdots,\{v_{\ell-1},v_\ell\}$ are present in $G(n,\lambda/n)$. Then, the hopcount is $H(p)=\ell$ and the total weight $L(p)$ is the sum of the weights assigned to $\{v_0,v_1\},\hdots,\{v_{\ell-1},v_\ell\}$.

\begin{lemma} \label{lem:first_moment}
For $\Psi_n$ defined in \eqref{eq:def_psi_n} and all $z,u\in\mathbb{R}$,
\begin{equation}\label{eqn:limit_first_moment_I}
\lim_{n\to\infty} \tau_n \mathbb{E}[\Psi_n((-\infty,z]\times(-\infty,u])] = \Lambda((-\infty,z]\times(-\infty,u])
\end{equation}
and
\begin{equation}\label{eqn:limit_first_moment_II}
\lim_{n\to\infty} \tau_n \mathbb{E}[\Psi_n(\mathbb{R}\times(-\infty,u])] = \Lambda(\mathbb{R}\times(-\infty,u]),
\end{equation}
where $\tau_n$ and $\Lambda$ are defined by \eqref{eq:def_tau_n} and \eqref{eq:def_Lambda}, respectively.
\end{lemma}

\begin{proof}
For $n\in\mathbb{N}$ we recall the definition \eqref{eq:def_of_k_n(z)} of $k_n(z)$. Note that
\begin{align*}
& \mathbb{E}[\Psi_n((-\infty,z]\times(-\infty,u])] \\
& = \mathbb{E} \sum_{\ell=1}^{k_n(z)} \sum_{(v_0,\hdots,v_\ell)\in \mathscr{P}_1^n } \mathbf{1}\Big\{v_0\leftrightarrow v_1 \leftrightarrow \hdots \leftrightarrow v_{\ell-1}\leftrightarrow v_\ell, \sum_{i=1}^\ell X_{\{v_{i-1},v_i\}} \leq \varrho_n + u \Big\}.
\end{align*}
For fixed vertices $v_0,\hdots,v_\ell$ we have, by the construction of the Erd\H{o}s-R\'enyi graph and its independence from the weights,
\begin{align*}
& \mathbb{P}\Big(v_0\leftrightarrow v_1 \leftrightarrow \hdots \leftrightarrow v_{\ell-1}\leftrightarrow v_\ell, \sum_{i=1}^\ell X_{\{v_{i-1},v_i\}} \leq \varrho_n + u \Big) \\
& = \mathbb{P}(v_0\leftrightarrow v_1 \leftrightarrow \hdots \leftrightarrow v_{\ell-1}\leftrightarrow v_\ell) \mathbb{P}\Big(\sum_{i=1}^\ell X_{\{v_{i-1},v_i\}} \leq \varrho_n + u \Big) = \bigg( \frac{\lambda}{n}\bigg)^{\ell} \mathbb{P}\Big(S_\ell \leq  \varrho_n + u \Big),
\end{align*}
where $S_\ell$ is the sum of $\ell$ independent copies of $X$. With the observation that there are $(n-2)_{\ell-1}$ ways to choose $v_1,\hdots,v_{\ell-1}$ (as $v_0=1$ and $v_\ell=n$), this yields
$$
\mathbb{E}[\Psi_n((-\infty,z]\times(-\infty,u])] = \sum_{\ell=1}^{k_n(z)} \frac{(n-2)_{\ell-1}}{n^{\ell}} \lambda^\ell \mathbb{P}\left(S_\ell \leq  \varrho_n + u\right).
$$
Together with the observations that, for $\ell\in \{1,\hdots, k_n(z)\}$ and $k_n(z)\leq n$,
$$
1\geq \frac{(n-2)_{\ell-1}}{n^{\ell-1}} \geq \frac{(n-2)_{k_n(z)-1}}{n^{k_n(z)-1}} \geq \bigg( 1 - \frac{k_n(z)}{n} \bigg)^{k_n(z)-1} \, \text{ and} \,  \bigg( 1 - \frac{k_n(z)}{n} \bigg)^{k_n(z)-1} \to1
$$
as $n\to\infty$, where we used $\lim_{n\to\infty} k_n(z)^2/n=0$, this implies
$$
\lim_{n\to\infty} \tau_n \mathbb{E}[\Psi_n((-\infty,z]\times(-\infty,u])] = \lim_{n\to\infty} \frac{\tau_n}{n}  \sum_{\ell=1}^{k_n(z)} \lambda^\ell \mathbb{P}\left(S_\ell \leq  \varrho_n + u \right).
$$
Now Lemma \ref{lem:renewal} b) completes the proof of \eqref{eqn:limit_first_moment_I}. 

To derive \eqref{eqn:limit_first_moment_II}, we can replace $k_n(z)$ by $\infty$ everywhere in the steps above until we bound $(n-2)_{\ell-1}/n^{\ell-1}$. This leads to
$$
\tau_n \mathbb{E}[\Psi_n(\mathbb{R}\times(-\infty,u])] \leq \frac{\tau_n}{n}  \sum_{\ell=1}^\infty \lambda^\ell \mathbb{P}\left(S_\ell \leq  \varrho_n + u \right)
$$
so that, by \eqref{eqn:V_infinity_non_arithmetic} and \eqref{eqn:V_infinity_arithmetic},
$$
\limsup_{n\to\infty} \tau_n \mathbb{E}[\Psi_n(\mathbb{R}\times(-\infty,u])] \leq \Lambda(\mathbb{R}\times(-\infty,u]).
$$
Together with \eqref{eqn:limit_first_moment_I} and the monotonicity in $z$ there, this proves \eqref{eqn:limit_first_moment_II}.
\end{proof}

To compute higher-order factorial moments of $\Psi_n$, one would have to study the joint behaviour of several short paths. This is difficult since the paths could intersect each other in an intricate way. In order to avoid this issue, we next introduce a point process only taking so-called uncrossed paths into account.

We say that a path crosses another path if there exists a maximal joint segment including neither $1$ nor $n$. Here, maximal means that the segment is not part of a larger joint segment and even a single vertex is considered as a segment. For examples of crossing and non-crossing paths we refer to Figure \ref{fig:crossings}.

\begin{figure}
\begin{center}
\begin{subfigure}{.5\textwidth}
\centering
\begin{tikzpicture}[scale=0.75]

\node[shape=circle,draw=black, minimum size=0.4cm] (Node1) at (0,0) {{\tiny $1$}};

\node[shape=circle,draw=black, minimum size=0.4cm] (Node11) at (1,1) {};

\node[shape=circle,draw=black, minimum size=0.4cm] (Node12) at (1,-1) {};

\node[shape=circle,draw=black, minimum size=0.4cm] (Node111) at (2,1.5) {};

\node[shape=circle,draw=black, minimum size=0.4cm] (Node112) at (2,0.5) {};

\node[shape=circle,draw=black, minimum size=0.4cm] (Node121) at (2,-1) {};

\node[shape=circle,draw=black, minimum size=0.4cm] (Node1211) at (3,-0.5) {};

\node[shape=circle,draw=black, minimum size=0.4cm] (Node1212) at (3,-1.5) {};

\node[shape=circle,draw=black, minimum size=0.4cm] (Node1111) at (5,1.5) {};

\node[shape=circle,draw=black, minimum size=0.4cm] (Node1121) at (5,0.5) {};

\node[shape=circle,draw=black, minimum size=0.4cm] (Node12111) at (5,-0.5) {};

\node[shape=circle,draw=black, minimum size=0.4cm] (Node12121) at (5,-1.5) {};

\node[shape=circle,draw=black, minimum size=0.4cm] (Node11111) at (6,1) {};

\node[shape=circle,draw=black, minimum size=0.4cm] (Node121211) at (6.5,-1) {};

\node[shape=circle,draw=black, minimum size=0.4cm] (Node111111) at (7,0.5) {};

\node[shape=circle,draw=black, minimum size=0.4cm] (Noden) at (8,0.5) {{\tiny $n$}};

\path [-, color=green, line width=0.05cm] (Node1.35) edge node[left] {} (Node11.235);

\path [-, color=green, line width=0.05cm] (Node11) edge node[left] {} (Node112);

\path [-, color=green, line width=0.05cm, dashed] (Node112) edge node[left] {} (Node1121);

\path [-, color=green, line width=0.05cm] (Node1121) edge node[left] {} (Node11111);

\path [-, color=green, line width=0.05cm] (Node11111.-25) edge node[left] {} (Node111111.165);

\path [-, color=green, line width=0.05cm] (Node111111) edge node[left] {} (Noden.180);

\path [-, color=blue, line width=0.05cm] (Node1.55) edge node[left] {} (Node11.215);

\path [-, color=blue, line width=0.05cm] (Node11) edge node[left] {} (Node111);

\path [-, color=blue, line width=0.05cm, dashed] (Node111) edge node[left] {} (Node1111);

\path [-, color=blue, line width=0.05cm] (Node1111) edge node[left] {} (Node11111);

\path [-, color=blue, line width=0.05cm] (Node11111.-5) edge node[left] {} (Node111111.145);

\path [-, color=blue, line width=0.05cm] (Node111111.20) edge node[left] {} (Noden.160);

\path [-, color=cyan, line width=0.05cm] (Node1.-35) edge node[left] {} (Node12.125);

\path [-, color=cyan, line width=0.05cm] (Node12.10) edge node[left] {} (Node121.170);

\path [-, color=cyan, line width=0.05cm] (Node121) edge node[left] {} (Node1211);

\path [-, color=cyan, line width=0.05cm, dashed] (Node1211) edge node[left] {} (Node12111);

\path [-, color=cyan, line width=0.05cm] (Node12111) edge node[left] {} (Node111111);

\path [-, color=cyan, line width=0.05cm] (Node111111.-20) edge node[left] {} (Noden.200);

\path [-, color=red, line width=0.05cm] (Node1.-55) edge node[left] {} (Node12.145);

\path [-, color=red, line width=0.05cm] (Node12.-10) edge node[left] {} (Node121.190);

\path [-, color=red, line width=0.05cm] (Node121) edge node[left] {} (Node1212);

\path [-, color=red, line width=0.05cm, dashed] (Node1212) edge node[left] {} (Node12121);

\path [-, color=red, line width=0.05cm] (Node12121) edge node[left] {} (Node121211);

\path [-, color=red, line width=0.05cm] (Node121211) edge node[left] {} (Noden);

\end{tikzpicture}
\end{subfigure}%
\begin{subfigure}{.5\textwidth}
\centering
\begin{tikzpicture}[scale=0.75]

\node[shape=circle,draw=black, minimum size=0.4cm] (Node1) at (0,0) {{\tiny $1$}};

\node[shape=circle,draw=black, minimum size=0.4cm] (Node11) at (1,1) {};

\node[shape=circle,draw=black, minimum size=0.4cm] (Node12) at (1,-0.5) {};

\node[shape=circle,draw=black, minimum size=0.4cm] (Node111) at (2,1.5) {};

\node[shape=circle,draw=black, minimum size=0.4cm] (Node112) at (2,0.5) {};

\node[shape=circle,draw=black, minimum size=0.4cm] (Node121) at (2,-0.5) {};

\node[shape=circle,draw=black, minimum size=0.4cm] (Node112l) at (3,0.5) {};

\node[shape=circle,draw=black, minimum size=0.4cm] (Node112r) at (4,0.5) {};

\node[shape=circle,draw=black, minimum size=0.4cm] (Node1111) at (5,1.5) {};

\node[shape=circle,draw=black, minimum size=0.4cm] (Node1121) at (5,0.5) {};

\node[shape=circle,draw=black, minimum size=0.4cm] (Node12111) at (5,-0.5) {};

\node[shape=circle,draw=black, minimum size=0.4cm] (Node11111) at (6,1) {};

\node[shape=circle,draw=black, minimum size=0.4cm] (Node111111) at (7,0.5) {};

\node[shape=circle,draw=black, minimum size=0.4cm] (Noden) at (8,0.5) {{\tiny $n$}};

\path [-, color=green, line width=0.05cm] (Node1.35) edge node[left] {} (Node11.235);

\path [-, color=green, line width=0.05cm] (Node11) edge node[left] {} (Node112);

\path [-, color=green, line width=0.05cm, dashed] (Node112) edge node[left] {} (Node112l);

\path [-, color=green, line width=0.05cm] (Node112l.10) edge node[left] {} (Node112r.170);

\path [-, color=green, line width=0.05cm, dashed] (Node112r) edge node[left] {} (Node1121);

\path [-, color=green, line width=0.05cm] (Node1121) edge node[left] {} (Node11111);

\path [-, color=green, line width=0.05cm] (Node11111.-25) edge node[left] {} (Node111111.165);

\path [-, color=green, line width=0.05cm] (Node111111) edge node[left] {} (Noden.180);

\path [-, color=blue, line width=0.05cm] (Node1.55) edge node[left] {} (Node11.215);

\path [-, color=blue, line width=0.05cm] (Node11) edge node[left] {} (Node111);

\path [-, color=blue, line width=0.05cm, dashed] (Node111) edge node[left] {} (Node1111);

\path [-, color=blue, line width=0.05cm] (Node1111) edge node[left] {} (Node11111);

\path [-, color=blue, line width=0.05cm] (Node11111.-5) edge node[left] {} (Node111111.145);

\path [-, color=blue, line width=0.05cm] (Node111111.20) edge node[left] {} (Noden.160);

\path [-, color=red, line width=0.05cm] (Node1) edge node[left] {} (Node12);

\path [-, color=red, line width=0.05cm] (Node12) edge node[left] {} (Node121);

\path [-, color=red, line width=0.05cm, dashed] (Node121) edge node[left] {} (Node112l);

\path [-, color=red, line width=0.05cm] (Node112l.-10) edge node[left] {} (Node112r.190);

\path [-, color=red, line width=0.05cm, dashed] (Node112r) edge node[left] {} (Node12111);

\path [-, color=red, line width=0.05cm] (Node12111) edge node[left] {} (Node111111);

\path [-, color=red, line width=0.05cm] (Node111111.-20) edge node[left] {} (Noden.200);

\end{tikzpicture}
\end{subfigure}
\caption{In the left panel, no paths cross each other. In contrast, the two lower paths (in green and red) cross in the panel on the right.}
\label{fig:crossings}
\end{center}
\end{figure}

For $a,w\in\mathbb{R}$ define an $(a,w)$-uncrossed path to be a path between $1$ and $n$ such that for any other path between $1$ and $n$ whose hopcount and total weight do not exceed $\sqrt{\beta\log(n)} a +\gamma \log(n)$ and $\frac{1}{\alpha} \log(n) + w$, respectively, any maximal joint segment of these two paths contains either $1$ or $n$. We introduce the point process
\begin{equation}\label{eq:def_Psi_n_a,w}
\Psi_n^{(a,w)} = \sum\limits_{p \in \mathcal{P}_1^n,\, p \text{ is } (a,w)\text{-uncrossed}} \delta_{\left(\frac{H(p) - \gamma \log(n)}{\sqrt{\beta \log(n)}}, L(p) - \varrho_n \right)}.
\end{equation}
Our next lemma allows us to approximate the original point process $\Psi_n$ by the new point process $\Psi_n^{(a,w)}$, which we will focus on in the sequel.

\begin{lemma} \label{lem:good_paths} For all $a,w,z,u \in \R$ with $z\leq a$ and $u\leq w$,
 $$
\lim_{n\to\infty} \E\left[\big|\Psi_n((-\infty, z] \times (-\infty,u]) - \Psi_n^{(a,w)}((-\infty, z] \times (-\infty,u]) \big| \right] = 0,
 $$
where $\Psi_n$ is defined by \eqref{eq:def_psi_n}.
\end{lemma}

\begin{proof}
Since every point of $\Psi^{(a,w)}_n$ is also a point of $\Psi_n$, we may drop the absolute value in the expectation. We have that
\begin{align*}
& \E\left[\Psi_n((-\infty,z] \times (-\infty,u] ) - \Psi_n^{(a,w)}((-\infty,z] \times (-\infty,u]) \right] \\
& = \mathbb{E} \sum\limits_{p \in \mathcal{P}_1^n} \mathbf{1}\Big\{ H(p)\leq k_n(z), L(p) \leq \varrho_n +u,  p \text{ is not } (a,w)\text{-uncrossed}\Big\} \\
& \leq \mathbb{E}\sum\limits_{(p, \tilde{p}) \in (\mathcal{P}_1^n)^2_{\neq}} \mathbf{1}\Big\{ H(p),H(\tilde{p})\leq k_n(a), L(p), L(\tilde{p}) \leq \varrho_n +w,  \tilde{p} \text{ crosses } p\Big\} \\
& =  \sum\limits_{(p, \tilde{p}) \in (\mathscr{P}_1^n)^2_{\neq}} \mathbb{P}\Big( H(p),H(\tilde{p})\leq k_n(a), L(p), L(\tilde{p}) \leq \varrho_n +w,  \tilde{p} \text{ crosses } p\Big). 
\end{align*}
In the following we let $p$ and $\tilde{p}$ have hopcounts $\ell\leq k_n(a)$ and $\tilde{\ell}\leq k_n(a)$, respectively, and assume that they have exactly $\overline{\ell}$ joint edges. Note that $\overline{\ell}\leq \min\{\ell,\tilde{\ell}\}-1$ as one of the paths cannot contain all edges of the other path. Moreover, we suppose that $\tilde{p}$ leaves and joins $p$ in total $q\geq 2$ times. The number of potential paths $p\in\mathscr{P}_1^n$ with hopcount $\ell$ is bounded by $n^{\ell-1}$. In order to find a path $\tilde{p}\in\mathscr{P}_1^n$ for given $p$, we have to choose $q$ pairs of vertices where $\tilde{p}$ leaves and joins $p$ and to determine their positions in $p$. This number of choices is bounded by $k_n(a)^{4q}$. Since then $\tilde{\ell}-\overline{\ell}-q$ vertices remain to choose, the total number of paths $\tilde{p}$ for given $p$ is bounded by $k_n(a)^{4q} n^{\tilde{\ell}-\overline{\ell}-q}$. The number of choices for $p$ and $\tilde{p}$ is, thus, bounded by $k_n(a)^{4q} n^{\ell+\tilde{\ell}-\overline{\ell}-q-1}$. Since they involve $\ell+\tilde{\ell}-\overline{\ell}$ edges, the probability to observe paths $p$ and $\tilde{p}$ in $G(n,\lambda/n)$ is $(\lambda/n)^{\ell+\tilde{\ell}-\overline{\ell}}$. Moreover, we have that
$$
\mathbb{P}\Big( L(p), L(\tilde{p}) \leq \varrho_n +w\Big) = \mathbb{P}\Big(S_{\ell-\overline{\ell}} + \overline{S}_{\overline{\ell}}, \widetilde{S}_{\tilde{\ell}-\overline{\ell}} + \overline{S}_{\overline{\ell}} \leq \varrho_n + w\Big),
$$
where $S_{\ell-\overline{\ell}}$, $\overline{S}_{\overline{\ell}}$ and $\widetilde{S}_{\tilde{\ell}-\overline{\ell}}$ are independent and have the distributions of the sums of $\ell-\overline{\ell}$, $\overline{\ell}$ and $\tilde{\ell}-\overline{\ell}$ independent copies of $X$, respectively. Here one can think of $\overline{S}_{\overline{\ell}}$ as the sum of weights of shared edges, while $S_{\ell-\overline{\ell}}$ and $\widetilde{S}_{\tilde{\ell}-\overline{\ell}}$ represent the sums of weights of edges only being part of either $p$ or $\tilde{p}$. Combining the considerations from above we see that
\begin{align*}
& \E\left[\Psi_n((-\infty,z] \times (-\infty,u]) - \Psi_n^{(a,w)}((-\infty,z] \times (-\infty,u]) \right] \\
& \leq \sum_{\ell=1}^{k_n(a)} \sum_{\tilde{\ell}=1}^{k_n(a)} \sum_{\overline{\ell}=0}^{\min\{\ell,\tilde{\ell}\}-1} \sum_{q=2}^{k_n(a)} k_n(a)^{4q} n^{\ell+\tilde{\ell}-\overline{\ell}-q-1} \frac{\lambda^{\ell+\tilde{\ell}-\overline{\ell}}}{n^{\ell+\tilde{\ell}-\overline{\ell}}} \mathbb{P}\left(S_{\ell-\overline{\ell}} + \overline{S}_{\overline{\ell}}, \widetilde{S}_{\tilde{\ell}-\overline{\ell}} + \overline{S}_{\overline{\ell}} \leq \varrho_n + w\right) \\
& \leq \frac{1}{n} \sum_{\ell=1}^{k_n(a)} \sum_{\tilde{\ell}=1}^{k_n(a)} \sum_{\overline{\ell}=0}^{\min\{\ell,\tilde{\ell}\}-1} \sum_{q=2}^{k_n(a)} \bigg(\frac{k_n(a)^4}{n}\bigg)^{q} \lambda^{\ell+\tilde{\ell}-\overline{\ell}} \mathbb{P}\left(S_{\ell-\overline{\ell}} + \overline{S}_{\overline{\ell}}, \widetilde{S}_{\tilde{\ell}-\overline{\ell}} + \overline{S}_{\overline{\ell}} \leq \varrho_n + w\right) \\
& \leq \frac{k_n(a)^8}{n-k_n(a)^4} \frac{1}{n^2} \sum_{j_1=1}^\infty \sum_{j_2=1}^\infty \sum_{j_3=0}^\infty \lambda^{j_1+j_2+j_3} \mathbb{P}\left(S_{j_1}+\overline{S}_{j_3}, \widetilde{S}_{j_2}+\overline{S}_{j_3} \leq \varrho_n + w\right) \\
& \leq \frac{k_n(a)^8}{n-k_n(a)^4} \mathbb{E} \sum_{j_3=0}^\infty \lambda^{j_3}  \frac{1}{n} \sum_{j_1=1}^\infty \lambda^{j_1} \mathbb{P}\left( S_{j_1}\leq \varrho_n + w - \overline{S}_{j_3} | \overline{S}_{j_3}\right) \\
& \hspace{4cm} \times \frac{1}{n} \sum_{j_2=1}^\infty \lambda^{j_2} \mathbb{P}\left( \widetilde{S}_{j_2}\leq \varrho_n + w - \overline{S}_{j_3} | \overline{S}_{j_3}\right).
\end{align*}
It follows from Lemma \ref{lem:renewal} a) that there exists a constant $C$ such that the right-hand side is bounded by
\begin{align*}
\frac{e^{2\alpha \varrho_n}}{n^2} \frac{ C^2 k_n(a)^8}{n-k_n(a)^4}\sum_{j_3=0}^\infty \lambda^{j_3} \mathbb{E}[ e^{2\alpha w - 2\alpha \widetilde{S}_{j_3}}] & = \frac{e^{2\alpha \varrho_n}}{n^2} \frac{ C^2 e^{2\alpha w} k_n(a)^8}{n-k_n(a)^4}\sum_{j=0}^\infty \big(\lambda \mathbb{E}[e^{- 2\alpha X}] \big)^j \\
& = \frac{e^{2\alpha \varrho_n}}{n^2} \frac{ C^2 e^{2\alpha w} k_n(a)^8}{n-k_n(a)^4} \frac{1}{1-\lambda \mathbb{E}[e^{- 2\alpha X}]},
\end{align*}
where we used $\lambda\E[e^{-2\alpha X}]<1$. Since $e^{\alpha \varrho_n}/n$ is bounded, the right-hand side converges to zero as $n\to\infty$.
\end{proof}

Lemma \ref{lem:good_paths} basically ensures that when studying several short paths it is sufficient to assume that they are joint at the beginning and the end and run separately in between. A similar observation is made in \cite{hof24} for some variance calculations for the number of paths with a fixed hopcount in inhomogeneous random graphs, where only pairs of paths with this behaviour are asymptotically relevant (see also \cite{bhk15,hk17} for similar arguments for configuration models). We emphasise that these paths counting arguments differ significantly from ours as they do not include any weights and cover only paths with a fixed hopcount.

\subsection{Some random variables associated with trees}\label{sec:trees}

In the following, we are interested in non-crossing paths as in the left panel of Figure \ref{fig:crossings} where the edges in the neighbourhoods of $1$ and $n$ that are used by more than one path form trees. To describe such situations, we introduce in this subsection marked trees and associated random variables. The latter will arise in the limits of the factorial moments of $\Psi_n^{(a,w)}$ in the next subsection. Here, we also show that they coincide with the moments of $W$.


Define, for $r \ge 2$, the set $\mathcal{T}_r$ of finite marked trees such that marks of all vertices are subsets of $[r]$ with a cardinality of at least $2$, the mark of the root is $[r]$, and marks of children are pairwise disjoint subsets of the mark of their parent. Note that this definition allows for a vertex to have exactly one child that inherits its exact mark. Figure \ref{fig:trees} shows two elements of $\mathcal{T}_4$.

\begin{figure}
\begin{center}
\begin{tikzpicture}[scale=0.7]

\node[shape=circle,draw=black, minimum size=1cm] (NodeTop) at (0,0) {{\footnotesize $1, 2, 3, 4$}};

\node[shape=circle,draw=black, minimum size=1cm] (NodeLeft) at (-2,-2) {{\footnotesize $1, 2$}};

\node[shape=circle,draw=black, minimum size=1cm] (NodeRight) at (2,-2) {{\footnotesize $3, 4$}};

\node[shape=circle,draw=black, minimum size=1cm] (NodeLeftLeft) at (-2,-4.5) {{\footnotesize $1, 2$}};

\path [-, line width=0.05cm] (NodeTop) edge node[left] {} (NodeLeft);

\path [-, line width=0.05cm] (NodeTop) edge node[left] {} (NodeRight);

\path [-, line width=0.05cm] (NodeLeft) edge node[left] {} (NodeLeftLeft);

\node[shape=circle,draw=black, minimum size=1cm] (NodeTopII) at (8,0) {{\footnotesize $1, 2, 3, 4$}};

\node[shape=circle,draw=black, minimum size=1cm] (NodeCenterII) at (8,-2.5) {{\footnotesize $2, 3, 4$}};

\node[shape=circle,draw=black, minimum size=1cm] (NodeBottomII) at (8,-5) {{\footnotesize $3, 4$}};

\path [-, line width=0.05cm] (NodeTopII) edge node[left] {} (NodeCenterII);

\path [-, line width=0.05cm] (NodeCenterII) edge node[left] {} (NodeBottomII);

\end{tikzpicture}

\caption{Two trees from $\mathcal{T}_4$ describing the behaviour of the non-crossing paths from the left panel of Figure \ref{fig:crossings} in the neighbourhoods of $1$ and $n$.}

\label{fig:trees}
\end{center}
\end{figure}

For $r$ non-crossing paths from $1$ to $n$ two trees from $\mathcal{T}_r$ arise in the following way. They contain all vertices in the neigbourhood of $1$ or $n$, respectively, that are traversed by more than one path and the marks of a vertex are the numbers of these paths. The trees in Figure \ref{fig:trees} describe the behaviour of the paths from the left panel of Figure \ref{fig:crossings} in the neighbourhoods of $1$ (left) and $n$ (right), where the paths are numbered from the bottom to the top.

For a given tree $T\in \mathcal{T}_r$ with $r\geq 2$ we associate with every edge an independent copy of $X$. For $I\subseteq [r]$ with $|I|\geq 2$ we denote by $S_I$ the sum of such random variables belonging to edges whose vertex with the larger distance to the root has mark $I$. 

The following lemma, where $\alpha$ and $\lambda$ are as throughout the rest of the paper, ensures the existence of some sums associated with the marked trees that will play a crucial role in the sequel and provides an important recursive formula.

\begin{lemma} \label{lem:recursion_Mr}
For $r\geq 2$ the series
\begin{equation} \label{eq:definition_Mr}
M^{(r)} := \sum_{T \in \mathcal{T}_r} \lambda^{|T|-1} \E\left[e^{-\alpha \sum\limits_{I \subseteq [r], \, |I|\ge 2} |I| S_I} \right]
\end{equation}
converges. Moreover, one has for any $r\ge 2$ the recursion
\begin{equation} \label{eq:recursion_Mr}
M^{(r)} = \frac{1}{1-\lambda \E\left[e^{-r \alpha X} \right]} \sum\limits_{\pi \in \Pi(r) \setminus \{[r]\}} \prod_{I \in \pi} \lambda \E\left[e^{-\alpha |I] X}\right] M^{(|I|)}
\end{equation}
with $M^{(1)}:=1$.
\end{lemma}

\begin{proof}
We first prove that the series in \eqref{eq:definition_Mr} is (absolutely) convergent. For a tree $T\in\mathcal{T}_r$ and $I\subseteq [r]$ with $|I| \ge 2$ let $N(T,I)$ be the number of vertices of $T$ with mark $I$ that are different from the root. By the independence of the random variables associated with the edges, we have 
$$
\E\left[e^{-\alpha \sum\limits_{I \subseteq [r], \, |I|\ge 2} |I| S_I} \right] = \prod_{I\subseteq [r], |I|\geq 2} \E\left[e^{-\alpha |I| S_I} \right] = \prod_{I\subseteq [r], |I|\geq 2} \E\left[e^{-\alpha |I| X} \right]^{N(T,I)}.
$$
Note that a tree $T\in\mathcal{T}_r$ is uniquely determined by $(N(T,I))_{I\subseteq[r], |I|\geq 2}$, although there does not necessarily exist a tree for any particular choice of these numbers. Thus, we obtain
\begin{align*}
\sum_{T \in \mathcal{T}_r} \lambda^{|T|-1} \E\left[e^{-\alpha \sum\limits_{I \subseteq [r], \, |I|\ge 2} |I| S_I} \right] &\leq \prod_{I\subseteq [r], |I|\geq 2} \sum_{j=0}^\infty \lambda^j \E\left[e^{-\alpha |I| X} \right]^{j}\\
& \leq \bigg( \sum_{j=0}^\infty \lambda^j \E\left[e^{-2\alpha X} \right]^{j}  \bigg)^{2^r}.
\end{align*}
Because $\lambda \E\big[ e^{-2\alpha X}\big]<1$ the geometric series on the right-hand side converges, whence the series in \eqref{eq:definition_Mr} is also convergent.

Next we establish the recursive relation \eqref{eq:recursion_Mr}. Recall that a tree $T\in\mathcal{T}_r$ has $N(T,[r])$ vertices with mark $[r]$ that are different from the root. We consider the vertex with mark $[r]$ that does not have a child with mark $[r]$. The marks of its children form a subpartition $\pi^*$ of $[r]$ where each block has at least the size two and at most $r-1$. We write $\Pi^*_{\geq 2}(r)$ for these subpartitions of $[r]$. Note that in contrast to a partition the union of all blocks of a subpartition does not need to be the whole set. For each $J\in \pi^*$ the child with mark $J$ is the root of a tree that is isomorphic to a tree from $\mathcal{T}_{|J|}$. Thus every tree $T\in\mathcal{T}_r$ is given by the length of the path at the beginning, whose vertices all have mark $[r]$, a subpartition $\pi^*\in\Pi^*_{\geq 2}(r)$ and trees $T_J\in\mathcal{T}_{|J|}$ for $J\in\pi^*$. This implies
\begin{align*}
M^{(r)} & = \sum_{T \in \mathcal{T}_r} \lambda^{|T|-1} \E\left[e^{-\alpha \sum\limits_{I \subseteq [r], \, |I|\ge 2} |I| S_I} \right] \\
& = \sum_{k=0}^\infty \lambda^k \E\big[e^{-\alpha r X} \big]^k \sum_{\pi^*\in\Pi_{\geq 2}^*(r)} \prod_{J\in\pi^*} \lambda \E\big[e^{-\alpha |J| X} \big] \sum_{T_J \in \mathcal{T}_{|J|}} \lambda^{|T_J|-1} \E\left[e^{-\alpha \sum\limits_{I \subseteq J, \, |I|\ge 2} |I| S_I} \right].
\end{align*}
Using the formula for the geometric series together with $\lambda \E\big[e^{-\alpha r X} \big]<1$ and the observation that the last sum is $M^{(|J|)}$, we obtain
$$
M^{(r)} = \frac{1}{1 - \lambda \E\big[e^{-\alpha r X} \big]} \sum_{\pi^*\in\Pi^*_{\geq 2}(r)} \prod_{J\in\pi^*} \lambda \E\big[e^{-\alpha |J| X} \big] M^{(|J|)}.
$$
We can extend each subpartition $\pi^*\in\Pi^*_{\geq 2}(r)$ to a partition $\pi\in\Pi(r)\setminus\{[r]\}$ by adding the remaining elements of $[r]$ as singletons. Since the additional factors in the product above are all one, we see that 
$$
M^{(r)} = \frac{1}{1 - \lambda \E\big[e^{-\alpha r X} \big]} \sum_{\pi\in\Pi(r)\setminus\{[r]\}} \prod_{J\in\pi} \lambda \E\big[e^{-\alpha |J| X} \big] M^{(|J|)},
$$
which completes the proof.
\end{proof}

The next result, which links the quantities $M^{(r)}$, $r\in\mathbb{N}$, and the moments of $W$, will be crucial see to that the limits of the factorial moments of $\Psi_n^{(a,w)}$ expressed in terms of $M^{(r)}$ are the same as the factorial moments of the limiting Cox processes derived in Section \ref{sec:moments_limit}, which involve moments of $W$.

\begin{coro}\label{cor:W_M_r}
Let $W$ be as in Lemma \ref{lem:exp_moments_W}. For any $r\in\mathbb{N}$,
$$
\E[W^r] = M^{(r)}.
$$
\end{coro}

\begin{proof}
By Lemma \ref{lem:recursion_moments_W} and Lemma \ref{lem:recursion_Mr}, $(\E[W^r])_{r\in\mathbb{N}}$ and $(M^{(r)})_{r\in\mathbb{N}}$ satisfy the same recursion. Together with $M^{(1)}=\E[W]=1$ this completes the proof.
\end{proof}

\subsection{Counting short paths: factorial moments}\label{sec:higher_moments}

In this subsection we compute higher order factorial moments of $\Psi_n^{(a,w)}$. Recall that two paths $p_1,p_2\in\mathscr{P}_1^n$ cross if they have a maximal joint segment containing neither $1$ nor $n$, i.e., if there exists a vertex $v$ that is neither $1$ nor $n$ and belongs to both $p_1$ and $p_2$ such that the segments from $1$ to $v$ in $p_1$ and $p_2$ are not identical and such that the segments from $v$ to $n$ in $p_1$ and $p_2$ are not identical. For more than two paths we say that they cross if any two of them cross. The next lemma is a crucial step for the computation of the factorial moments of $\Psi_n^{(a,w)}$.

\begin{lemma}\label{lem:not_crossing}
For $r\in\mathbb{N}$ and $z_1,\hdots,z_r, u_1,\hdots,u_r\in\mathbb{R}$,
\begin{align*}
& \lim_{n\to\infty} \tau_n^r \mathbb{E} \sum_{(p_1,\hdots,p_r)\in(\mathcal{P}_1^n)^r_{\neq}} \mathbf{1} \{p_1,\hdots,p_r \text{ do not cross and } H(p_i)\leq k_n(z_i)\text{ for } i\in\{1,\hdots,r\}\} \\
& \hspace{3cm} \times \mathbf{1}\left\{ L(p_i) \leq \varrho_n + u_i \text{ for } i\in\{1,\hdots,r\} \right\} \\
& = \big( M^{(r)} \big)^2 \prod_{i=1}^r \Lambda((-\infty,z_i]\times(-\infty,u_i])
\end{align*}
with $M^{(r)}$ as in Lemma \ref{lem:recursion_Mr}, and where $\varrho_n$, $\tau_n$, $\Lambda$ and $k_n$ are as defined in \eqref{eq:def_rho} and \eqref{eq:def_tau_n}--\eqref{eq:def_of_k_n(z)}, respectively.
\end{lemma}

\begin{proof}
For $r=1$ the statement follows from Lemma \ref{lem:first_moment}, whence we assume $r\geq 2$ for the rest of the proof. We have that
\begin{align*}
\mathscr{M}_n & := \tau_n^r \mathbb{E} \sum_{(p_1,\hdots,p_r)\in(\mathcal{P}_1^n)^r_{\neq}} \mathbf{1} \{p_1,\hdots,p_r \text{ do not cross and } H(p_i)\leq k_n(z_i)\text{ for } i\in\{1,\hdots,r\}\} \\
& \hspace{3cm} \times \mathbf{1}\left\{ L(p_i) \leq \varrho_n + u_i \text{ for } i\in\{1,\hdots,r\} \right\} \\
& = \tau_n^r \sum_{(p_1,\hdots,p_r)\in(\mathscr{P}_1^n)^r_{\neq}} \mathbf{1} \{p_1,\hdots,p_r \text{ do not cross and } H(p_i)\leq k_n(z_i)\text{ for } i\in\{1,\hdots,r\}\} \\
& \hspace{3cm} \times \mathbb{P}(p_1,\hdots,p_r\in\mathcal{P}_1^n) \mathbb{P}\left( L(p_i) \leq \varrho_n + u_i \text{ for } i\in\{1,\hdots,r\} \right).
\end{align*}
Next we use that $p_1,\hdots,p_r$ are distinct paths from $1$ to $n$ that do not cross each other. This means that any two of them run together at the beginning, then become disjoint and finally are again together. Note that every path $p_1,\hdots,p_r$ must have an edge that is not part of one of the other paths since otherwise two of the other paths cross each other. Thus, $p_1,\hdots,p_r$ first split, then they are all disjoint and finally they again rejoin (see Figure \ref{fig:crossings} for an example). The splitting and rejoining can be described by trees $T',T''\in\mathcal{T}_r$ as discussed in the previous subsection. More precisely, the subgraph that is induced by the shared edges and contains the vertex $1$ is isomorphic to $T'$, while the subgraph that is induced by the shared edges and contains the vertex $n$ is isomorphic to $T''$ (see Figure \ref{fig:trees} for the example of Figure \ref{fig:crossings}). The vertices corresponding to the vertices of $T'$ and $T''$ are denoted by $\mathbf{v}'=(v_1',\hdots,v_{|T'|}')$ and $\mathbf{v}''=(v_1'',\hdots,v_{|T''|}'')$, where $v_1'=1$ and $v_1''=n$ correspond to the roots of $T'$ and $T''$.

If a $v\in[n]$ corresponds to a vertex in $T'$ or $T''$ with marks $I\subseteq [r]$, this means that $v$ belongs to exactly the paths $p_i$ with $i\in I$. The $i$-th path becomes disjoint from all other paths in the vertices corresponding to the vertices of $T'$ and $T''$ that have $i$ as an element of their mark and maximise the distance to the root. For $i\in[r]$ let $\ell_i$ be the number of edges of $p_i$ that are not shared with other paths and let $\mathbf{v}^{(i)}=(v_1^{(i)},\hdots,v_{\ell_i-1}^{(i)})$ be the corresponding vertices, while $\ell_i'$ and $\ell_i''$ are the numbers of shared edges at the beginning and the end. By $v_0^{(i)}$ and $v_{\ell_i}^{(i)}$ we denote the vertices where $p_i$ starts and ends being disjoint from the other paths. The sums of the weights of shared edges of $p_i$ are $S_i'$ and $S_i''$ and the sum of the weights of the not shared edges is $S_i$. We use the abbreviations $\ell_0=|T'|$ and $\ell_{r+1}=|T''|$, where $|T'|$ and $|T''|$ are the numbers of vertices of $T'$ and $T''$.

Using the notation introduced above we see that
\begin{align*}
\mathscr{M}_n & = \tau_n^r \sum_{T',T''\in\mathcal{T}_r} \sum_{\ell_1=1}^{k_n(z_1)-\ell_1'-\ell_1''}\hdots\sum_{\ell_r=1}^{k_n(z_r)-\ell_r'-\ell_r''} \sum_{\substack{(\mathbf{v}',\mathbf{v}'',\mathbf{v}^{(1)},\hdots,\mathbf{v}^{(r)})\in [n]_{\neq}^{\ell_0+\hdots+\ell_{r+1}-r}\\ v_1'=1, v_1''=n}} \\
& \hspace{2cm} \times \mathbb{E} \mathbf{1}\{ v'_{i}\leftrightarrow v'_{j} \text{ for all } i,j\in\{1,\hdots,|T'|\} \text{ with $i\leftrightarrow j$ in $T'$}\} \\
& \hspace{2.75cm} \times \mathbf{1}\{ v''_{i}\leftrightarrow v''_{j} \text{ for all } i,j\in\{1,\hdots,|T''|\} \text{ with $i\leftrightarrow j$ in $T''$}\} \\
& \hspace{2.75cm} \times \mathbf{1}\{ v^{(i)}_0 \leftrightarrow v^{(i)}_1 \leftrightarrow \hdots\leftrightarrow v^{(i)}_{\ell_i-1} \leftrightarrow v^{(i)}_{\ell_i} \text{ for } i\in[r] \} \\
& \hspace{2.75cm} \times \mathbf{1}\left\{ S_i'+S_i''+S_i\leq \varrho_n + u_i \text{ for } i\in[r] \right\}. 
\end{align*}
Obviously the first three indicators under the expectation are independent from the fourth. The probability that all required edges are present in the Erd\H{o}s--R\'{e}nyi graph $G(n,\lambda/n)$ is given by $(\lambda/n)^{|T|'+|T''|-2+\sum_{j=1}^r \ell_j}$ since the exponent is the total number of edges occurring in the paths. Note that the probability of the event in the last indicator does not depend on the particular edges but only on $T',T''$ and $\ell_1,\hdots,\ell_r$. This yields
\begin{align*}
\mathscr{M}_n & = \tau_n^r \sum_{T',T''\in\mathcal{T}_r} \sum_{\ell_1=1}^{k_n(z_1)-\ell_1'-\ell_1''}\hdots\sum_{\ell_r=1}^{k_n(z_r)-\ell_r'-\ell_r''} (n-2)_{\ell_0+\hdots+\ell_{r+1}-2-r} \bigg( \frac{\lambda}{n} \bigg)^{\ell_0+\hdots+\ell_{r+1}-2} \\
& \hspace{4cm} \times  \mathbb{P}\left( S_i'+S_i''+S_i\leq \varrho_n + u_i \text{ for } i\in[r] \right)\\
& =  \sum_{T',T''\in\mathcal{T}_r} \sum_{\ell_1=1}^{k_n(z_1)-\ell_1'-\ell_1''}\hdots\sum_{\ell_r=1}^{k_n(z_r)-\ell_r'-\ell_r''} \frac{(n-2)_{\ell_0+\hdots+\ell_{r+1}-2-r}}{n^{\ell_0+\hdots+\ell_{r+1}-2-r}} \lambda^{|T'|+|T''|-2} \\
& \hspace{2.5cm} \times\mathbb{E} \prod_{i=1}^r \frac{\tau_n}{n} \lambda^{\ell_i} \mathbb{P}\left( S_i\leq \varrho_n + u_i - S_i' -S_i''| S_i',S_i'' \right).
\end{align*}
As the ratio of the descending factorial and the power of $n$ is at most one, for fixed $T',T''\in\mathcal{T}_r$ the inner sums are bounded by 
\begin{align*}
\lambda^{|T'|+|T''|-2} \mathbb{E} \prod_{i=1}^r \frac{\tau_n}{n} \sum_{\ell_i=1}^{\infty} \lambda^{\ell_i} \mathbb{P}\left( S_i\leq \varrho_n + u_i - S_i' -S_i''| S_i',S_i'' \right).
\end{align*}
From Lemma \ref{lem:renewal} a) we know that each of the factors in the product is at most
$$
\frac{\tau_n e^{\alpha\varrho_n}}{n} C e^{\alpha (u_i - S_i' - S_i'')} = C e^{\alpha (u_i - S_i' - S_i'')},
$$
for some constant $C$. Hence the summand for $T',T''\in \mathcal{T}_r$ is bounded by
\begin{align*}
& C^r e^{\alpha \sum_{i=1}^r u_i} \lambda^{|T'|+|T''|-2} \mathbb{E}\left[e^{-\alpha \sum_{i=1}^r S_i'+S_i''}\right] \\
& = C^r e^{\alpha \sum_{i=1}^r u_i} \lambda^{|T'|+|T''|-2} \mathbb{E}\left[ e^{-\alpha \sum_{I\subseteq[r], |I|\geq 2} |I| S_{T',I}}\right] \mathbb{E}\left[ e^{-\alpha \sum_{I\subseteq[r], |I|\geq 2} |I| S_{T'',I}}\right],
\end{align*}
where, for $T\in \{T',T''\}$ and $I\subseteq [r]$, $S_{T,I}$ is the sum of the weights of edges of $T$ whose vertex with the larger distance to the root has mark $I$. Thus, by \eqref{eq:definition_Mr} in Lemma \ref{lem:recursion_Mr}, the series over all $T',T''$ converges, whence we can apply the dominated convergence theorem to $\mathscr{M}_n$. For $T',T''\in\mathcal{T}_r$ we obtain
\begin{align*}
& \lim_{n\to\infty} \sum_{\ell_1=1}^{k_n(z_1)-\ell_1'-\ell_1''}\hdots\sum_{\ell_r=1}^{k_n(z_r)-\ell_r'-\ell_r''} \frac{(n-2)_{\ell_0+\hdots+\ell_{r+1}-2-r}}{n^{\ell_0+\hdots+\ell_{r+1}-2-r}} \lambda^{|T'|+|T''|-2} \\
& \hspace{4.5cm} \times\mathbb{E} \prod_{i=1}^r \frac{\tau_n}{n} \lambda^{\ell_i} \mathbb{P}\left( S_i\leq \varrho_n + u_i - S_i' -S_i''| S_i',S_i'' \right) \\
& = \lim_{n\to\infty} \sum_{\ell_1=1}^{k_n(z_1)-\ell_1'-\ell_1''}\hdots\sum_{\ell_r=1}^{k_n(z_r)-\ell_r'-\ell_r''} \mathbb{E} \prod_{i=1}^r \frac{\tau_n}{n} \lambda^{\ell_i} \mathbb{P}\left( S_i\leq \varrho_n + u_i - S_i' -S_i''| S_i',S_i'' \right),
\end{align*}
where we used that, for all $\ell_0,\hdots,\ell_{r+1}$ in the sums and $n$ sufficiently large,
\begin{align*}
1 & \geq \frac{(n-2)_{\ell_0+\hdots+\ell_{r+1}-2-r}}{n^{\ell_0+\hdots+\ell_{r+1}-2-r}} \\
& \geq \bigg( 1 - \frac{k_n(z_1)+\hdots+k_n(z_r)+\ell_0+\ell_{r+1}}{n}  \bigg)^{k_n(z_1)+\hdots+k_n(z_r)+\ell_0+\ell_{r+1}}
\end{align*}
and the right-hand side converges to one as $n\to\infty$. From Lemma \ref{lem:renewal} b) it follows that
\begin{align*}
\lim_{n\to\infty} & \frac{\tau_n}{n} \sum_{\ell_i=1}^{k_n(z_i)-\ell_i'-\ell_i''} \lambda^{\ell_i} \mathbb{P}\left( S_i\leq \varrho_n + u_i - S_i' -S_i''| S_i',S_i'' \right)\\
 & = \Lambda((-\infty,z_i]\times (-\infty,u_i-S_i'-S_i'']) \\
& = e^{-\alpha(S_i'+S_i'')} \Lambda((-\infty,z_i]\times (-\infty,u_i])
\end{align*}
for $i\in\{1,\hdots,r\}$. For the non-arithmetic case the latter identity is obvious, while in the arithmetic case one uses that $S_i'+S_i''$ belongs to $M\mathbb{Z}$. Thus, we obtain by the dominated convergence theorem and the definition of $M^{(r)}$ in \eqref{eq:definition_Mr} that
\begin{align*}
\lim_{n\to\infty} \mathscr{M}_n & = \sum_{T',T''\in\mathcal{T}_r} \lambda^{|T'|+|T''|-2} \mathbb{E} \prod_{i=1}^r e^{-\alpha(S_i'+S_i'')} \Lambda((-\infty,z_i]\times(-\infty,u_i]) \\
& = \sum_{T',T''\in\mathcal{T}_r} \lambda^{|T'|+|T''|-2} \mathbb{E} \Big[ e^{-\alpha \sum_{I\subseteq [r], |I|\geq 2} S_{T',I}+S_{T'',I}} \Big] \prod_{i=1}^r \Lambda((-\infty,z_i]\times(-\infty,u_i]) \\
& = \bigg( \sum_{T\in\mathcal{T}_r} \lambda^{|T|-1} \mathbb{E}\Big[ e^{-\alpha \sum_{I\subseteq [r], |I|\geq 2} S_{T,I}}\Big] \bigg)^2 \prod_{i=1}^r \Lambda((-\infty,z_i]\times(-\infty,u_i]) \\
& = \big( M^{(r)} \big)^2 \prod_{i=1}^r \Lambda((-\infty,z_i]\times (-\infty,u_i]) ,
\end{align*}
which completes the proof.
\end{proof}

We say that a path $\tilde{p}$ crosses paths $p_1,\hdots,p_r$ if $\tilde{p}$ crosses at least one of $p_1,\hdots,p_r$. The expression considered in the previous lemma is not a factorial moment of $\Psi_n^{(a,w)}$ since we do not require that $p_1,\hdots,p_r$ are all $(a,w)$-uncrossed. To take care of that issue, we control the situation that at least one of them is crossed by the following lemma.

\begin{lemma} \label{lem:contribution_crossed}
For $r\in\mathbb{N}$ and $z,u\in\mathbb{R}$,
\begin{align*}
& \lim_{n\to\infty} \mathbb{E} \sum_{(p_1,\hdots,p_r,\tilde{p})\in(\mathcal{P}_1^n)^{r+1}_{\neq}} \mathbf{1}\{p_1,\hdots,p_r\text{ do not cross},\, \tilde{p}\text{ crosses } p_1,\hdots,p_r\}\\
& \hspace{3.75cm} \times \mathbf{1}\{H(p_1),\hdots,H(p_r),H(\tilde{p})\leq k_n(z)\}\\
& \hspace{3.75cm} \times \mathbf{1}\left\{L(p_1),\hdots,L(p_r),L(\tilde{p})\leq \varrho_n +u\right\} =0,
\end{align*}
where $\varrho_n$ and $k_n$ are as defined in \eqref{eq:def_rho} and \eqref{eq:def_of_k_n(z)}, respectively.
\end{lemma}

\begin{proof}
For $r=1$ the statement was already shown in the proof of Lemma \ref{lem:good_paths}, whence we can assume $r\geq 2$ in the sequel. We have
\begin{align}
& \mathbb{E} \sum_{(p_1,\hdots,p_r,\tilde{p})\in(\mathcal{P}_1^n)^{r+1}_{\neq}} \mathbf{1}\{p_1,\hdots,p_r\text{ do not cross},\, \tilde{p}\text{ crosses } p_1,\hdots,p_r\} \notag\\
& \hspace{2.25cm} \times \mathbf{1}\left\{H(p_1),\hdots,H(p_r),H(\tilde{p})\leq k_n(z), L(p_1),\hdots,L(p_r),L(\tilde{p})\leq \varrho_n +u\right\} \notag\\
& = \sum_{(p_1,\hdots,p_r)\in(\mathscr{P}_1^n)^{r}_{\neq}} \mathbf{1}\{p_1,\hdots,p_r\text{ do not cross},\, H(p_1),\hdots,H(p_r)\leq k_n(z) \} \notag\\
& \hspace{1.25cm} \times \sum_{\tilde{p}\in\mathscr{P}_1^n\setminus\{p_1,\hdots,p_{r}\}} \mathbf{1}\left\{\tilde{p}\text{ crosses } p_1,\hdots,p_r, H(\tilde{p})\leq k_n(z) \right\} \mathbb{P}(p_1,\hdots,p_r,\tilde{p}\in\mathcal{P}_1^n) \notag \\
& \hspace{5cm} \times \mathbb{P}\left(L(p_1),\hdots,L(p_r),L(\tilde{p})\leq \varrho_n + u\right). \label{eqn:inner_sum}
\end{align}
In the following we bound the inner sum for given $p_1,\hdots,p_r$. We denote by $\tilde{\ell}$ the number of edges of $\tilde{p}$. Let $s_0,\hdots,s_q$ for $q\geq 1$ be the maximal segments of $\tilde{p}$ consisting of vertices and edges that also occur in $p_1,\hdots,p_r$. We write $\overline{\ell}$ for the sum of the numbers of edges of $s_0,\hdots,s_q$ and $\overline{S}_{\overline{\ell}}$ for the sum of their weights. In the following let $p_1,\hdots,p_r$ and $s_0,\hdots,s_q$ be fixed. Let $\mathscr{P}_1^n(s_0,\hdots,s_q)$ be the set of all paths $\tilde{p}\in\mathscr{P}_1^n$ whose maximal segments consisting of vertices and edges from $p_1,\hdots,p_r$ are $s_0,\hdots,s_q$. For $\tilde{p}\in \mathscr{P}_1^n(s_0,\hdots,s_q)$ we observe that
$$
\mathbb{P}(\tilde{p}\in\mathcal{P}_1^n| p_1,\hdots,p_r\in\mathcal{P}_1^n) = \bigg( \frac{\lambda}{n}\bigg)^{\tilde{\ell}-\overline{\ell}}
$$
and
$$
\mathbb{P}\left( L(\tilde{p})\leq \varrho_n + u | \overline{S}_{\overline{\ell}}\right) = \mathbb{P}\left( T_{\tilde{\ell}-\overline{\ell}} \leq \varrho_n + u - \overline{S}_{\overline{\ell}} | \overline{S}_{\overline{\ell}} \right),
$$
where $T_{\tilde{\ell}-\overline{\ell}}$ is the sum of $\tilde{\ell}-\overline{\ell}$ independent copies of $X$. Moreover, the number of all paths $\tilde{p}\in \mathscr{P}_1^n(s_0,\hdots,s_q)$ with $\tilde{\ell}$ edges can be bounded by $n^{\tilde{\ell}-1-\overline{\ell}-(q-1)}= n^{\tilde{\ell}-\overline{\ell}-q}$. Combining the previous estimates and Lemma \ref{lem:renewal} a) we obtain
\begin{align*}
& \sum_{\tilde{p}\in\mathscr{P}_1^n(s_0,\hdots,s_{q})} \mathbf{1}\left\{\tilde{p}\text{ crosses } p_1,\hdots,p_r, H(\tilde{p})\leq k_n(z) \right\} \mathbb{P}(\tilde{p}\in\mathcal{P}_1^n | p_1,\hdots,p_r\in\mathcal{P}_1^n)\\
&\hspace{5cm}\times \mathbb{P}\left(L(\tilde{p})\leq \varrho_n + u | \overline{S}_{\overline{\ell}} \right) \notag \\
& \leq \sum_{\tilde{\ell}=\overline{\ell}+q}^{k_n(z)} n^{\tilde{\ell}-\overline{\ell}-q} \bigg( \frac{\lambda}{n}\bigg)^{\tilde{\ell}-\overline{\ell}} \mathbb{P}\left( T_{\tilde{\ell}-\overline{\ell}} \leq \varrho_n + u - \overline{S}_{\overline{\ell}} | \overline{S}_{\overline{\ell}} \right) \leq \frac{1}{n^q} \sum_{\ell=1}^\infty \lambda^{\ell} \mathbb{P}\left( T_{\ell} \leq \varrho_n + u - \overline{S}_{\overline{\ell}} | \overline{S}_{\overline{\ell}} \right)  \\
& \leq \frac{C}{n^q} e^{\alpha (\varrho_n+u - \overline{S}_{\overline{\ell}})} \leq \frac{C}{n^{q-1}} e^{\alpha u} e^{-\alpha \overline{S}_{\overline{\ell}}}=:M_q(p_1,\hdots,p_r).
\end{align*}
For $q\geq2$ we have the upper bound
$$
M_q(p_1,\hdots,p_r)\le\frac{C e^{\alpha u}}{n^{q-1}}.
$$
The only situation that $\tilde{p}$ crosses one of $p_1,\hdots,p_r$ and $q=1$ is if $\tilde{p}$ leaves and joins in the part where paths run together (see Figure \ref{fig:q=1} for an example). Thus, $\tilde{p}$ must follow one of $p_1,\hdots,p_r$ to come from the multiple edges in the neighbourhood of $1$ to that in the neighbourhood of $n$ as otherwise we would have $q\geq 2$. The latter implies that $\overline{S}_{\overline{\ell}}\geq S_j$ for some $j\in[r]$, where $S_j$ is the sum of the weights of edges only included in $p_j$ but not in any other of the paths $p_i$, $i\neq j$. This yields
$$
M_q(p_1,\hdots,p_r)\leq C e^{\alpha u} \sum_{j\in[r]} e^{-\alpha S_j}
$$
for $q=1$.

\begin{figure}
\begin{center}
\begin{tikzpicture}[scale=0.75]

\node[shape=circle,draw=black, minimum size=0.4cm] (Node1) at (0,0) {{\tiny $1$}};

\node[shape=circle,draw=black, minimum size=0.4cm] (Node11) at (1,1) {};

\node[shape=circle,draw=black, minimum size=0.4cm] (Node12) at (1,-1) {};

\node[shape=circle,draw=black, minimum size=0.4cm] (Node121) at (2,-1) {};

\node[shape=circle,draw=black, minimum size=0.4cm] (Node111) at (2,1) {};

\node[shape=circle,draw=black, minimum size=0.4cm] (Node1111) at (3,1.5) {};

\node[shape=circle,draw=black, minimum size=0.4cm] (Node1112) at (3,0.5) {};

\node[shape=circle,draw=black, minimum size=0.4cm] (Node11111) at (5,1.5) {};

\node[shape=circle,draw=black, minimum size=0.4cm] (Node11121) at (5,0.5) {};

\node[shape=circle,draw=black, minimum size=0.4cm] (Node1211) at (5,-1) {};

\node[shape=circle,draw=black, minimum size=0.4cm] (Noden) at (6,0) {{\tiny $n$}};

\path [-, color=blue, line width=0.05cm] (Node1.55) edge node[left] {} (Node11.215);

\path [-, color=blue, line width=0.05cm] (Node11.20) edge node[left] {} (Node111.160);

\path [-, color=blue, line width=0.05cm] (Node111) edge node[left] {} (Node1111);

\path [-, color=blue, line width=0.05cm, dashed] (Node1111) edge node[left] {} (Node11111);

\path [-, color=blue, line width=0.05cm] (Node11111) edge node[left] {} (Noden);

\path [-, color=green, line width=0.05cm] (Node1.35) edge node[left] {} (Node11.235);

\path [-, color=green, line width=0.05cm] (Node11.0) edge node[left] {} (Node111.180);

\path [-, color=green, line width=0.05cm] (Node111.0) edge node[left] {} (Node1112.160);

\path [-, color=green, line width=0.05cm, dashed] (Node1112.10) edge node[left] {} (Node11121.170);

\path [-, color=green, line width=0.05cm] (Node11121.0) edge node[left] {} (Noden.160);

\path [-, color=cyan, line width=0.05cm] (Node1.-35) edge node[left] {} (Node12.125);

\path [-, color=cyan, line width=0.05cm] (Node12) edge node[left] {} (Node121);

\path [-, color=cyan, line width=0.05cm, dashed] (Node121) edge node[left] {} (Node1211);

\path [-, color=cyan, line width=0.05cm] (Node1211) edge node[left] {} (Noden);

\path [-, color=red, line width=0.05cm] (Node1.-55) edge node[left] {} (Node12.145);

\path [-, color=red, line width=0.05cm] (Node11.-20) edge node[left] {} (Node111.200);

\path [-, color=red, line width=0.05cm] (Node12) edge node[left] {} (Node11);

\path [-, color=red, line width=0.05cm] (Node111.-20) edge node[left] {} (Node1112.180);

\path [-, color=red, line width=0.05cm, dashed] (Node1112.-10) edge node[left] {} (Node11121.190);

\path [-, color=red, line width=0.05cm] (Node11121.-20) edge node[left] {} (Noden.180);

\end{tikzpicture}

\caption{In this example for $q=1$ the highest path (in blue) is crossed by the lowest path leaving $1$ (in red).}

\label{fig:q=1}
\end{center}
\end{figure}

Next we derive an upper bound for the number of choices for $s_0,\hdots,s_q$. For each of the $2q$ vertices that are starting and end points of $s_0,\hdots,s_q$ and distinct from $1$ and $n$ we have less than $r k_n(z)$ choices. For two given vertices on $p_1,\hdots,p_r$ the number of paths between the two vertices only using edges of $p_1,\hdots,p_r$ is bounded by a constant $c_r$ only depending on $r$. In fact the number of vertices where $p_1,\hdots,p_r$ leave and join each other is bounded by a constant depending on $r$ and a path between the two given vertices is completely determined by these vertices that are passed and their order. In total this gives us at most 
$$
c_r^{q+1} r^{2q} k_n(z)^{2q}
$$ 
possibilities to choose $s_0,\hdots,s_q$. Together with the bounds for $M_q(p_1,\hdots,p_r)$ this implies that
\begin{align*}
& \sum_{\tilde{p}\in\mathscr{P}_1^n\setminus\{p_1,\hdots,p_{r}\}} \mathbf{1}\left\{\tilde{p}\text{ crosses } p_1,\hdots,p_r, H(\tilde{p})\leq k_n(z) \right\} \mathbb{P}(p_1,\hdots,p_r,\tilde{p}\in\mathcal{P}_1^n) \notag \\
& \hspace{5cm} \times \mathbb{P}\left(L(p_1),\hdots,L(p_r),L(\tilde{p})\leq \varrho_n + u\right) \\
& \leq \mathbb{P}(p_1,\hdots,p_r\in\mathcal{P}_1^n)\hspace{-2.5pt} \sum_{q=1}^{k_n(z)} c_r^{q+1} r^{2q} k_n(z)^{2q} \mathbb{E}\left[M_q(p_1,\hdots,p_r) \mathbf{1}\left\{L(p_1),\hdots,L(p_r)\leq \varrho_n + u\right\} \right] \\
& \leq C e^{\alpha u} c_r^2 r^{2} k_n(z)^{2}  \mathbb{P}(p_1,\hdots,p_r\in\mathcal{P}_1^n) \mathbb{E}\left[ \sum_{j=1}^r e^{-\alpha S_j} \mathbf{1}\{L(p_1),\hdots,L(p_r)\leq \varrho_n + u\} \right] \\
& \quad  + C e^{\alpha u} \mathbb{P}(p_1,\hdots,p_r\in\mathcal{P}_1^n) \mathbb{P}\left(L(p_1),\hdots,L(p_r)\leq \varrho_n + u\right) \sum_{q=2}^{k_n(z)} \frac{c_r^{q+1} r^{2q} k_n(z)^{2q}}{n^{q-1}}.
\end{align*}
Now we put this bound into \eqref{eqn:inner_sum}. First we consider
\begin{align*}
J_1 := & C e^{\alpha u} c_r^2 r^{2} k_n(z)^{2}\hspace{-4.51pt} \sum_{(p_1,\hdots,p_r)\in(\mathscr{P}_1^n)^{r}_{\neq}} \mathbf{1}\{p_1,\hdots,p_r\text{ do not cross},\, H(p_1),\hdots,H(p_r)\leq k_n(z) \} \\
& \hspace{2.75cm} \times \mathbb{P}(p_1,\hdots,p_r\in\mathcal{P}_1^n) \mathbb{E}\left[ \sum_{j=1}^r e^{-\alpha S_j} \mathbf{1}\{L(p_1),\hdots,L(p_r)\leq \varrho_n + u\} \right] \\
= & C e^{\alpha u} c_r^2 r^{3} k_n(z)^{2}\hspace{-4.51pt}\sum_{(p_1,\hdots,p_r)\in(\mathscr{P}_1^n)^{r}_{\neq}} \mathbf{1}\{p_1,\hdots,p_r\text{ do not cross},\, H(p_1),\hdots,H(p_r)\leq k_n(z) \} \\
& \hspace{3cm} \times \mathbb{P}(p_1,\hdots,p_r\in\mathcal{P}_1^n) \mathbb{E}\left[e^{-\alpha S_r} \mathbf{1}\{L(p_1),\hdots,L(p_r)\leq \varrho_n + u\} \right]. 
\end{align*}
We can obtain the path $p_r$ by choosing the two vertices belonging to $p_1,\hdots,p_{r-1}$, where it starts and ends to be disjoint from these paths and connecting them via $i$ edges. Letting $T_i$ be the sum of $i$ independent copies of $X$ we obtain this way that $J_1$ is at most
\begin{align*}
& C e^{\alpha u} c_r^2 r^{3} k_n(z)^{2}\hspace{-8.4pt} \sum_{(p_1,\hdots,p_{r-1})\in(\mathscr{P}_1^n)^{r-1}_{\neq}} \hspace{-8.4pt}\mathbf{1}\{p_1,\hdots,p_{r-1}\text{ do not cross},\, H(p_1),\hdots,H(p_{r-1})\leq k_n(z) \} \\
& \hspace{4cm} \times \mathbb{P}(p_1,\hdots,p_{r-1}\in\mathcal{P}_1^n) \mathbb{P}(L(p_1),\hdots,L(p_{r-1})\leq \varrho_n + u) \\
& \hspace{4cm} \times (2+(r-1)k_n(z))^2 \sum_{i=1}^{k_n(z)} \frac{\lambda^i}{n} \mathbb{E}\left[\mathbf{1}\left\{T_i\leq \varrho_n+u\right\} e^{-\alpha T_i} \right].
\end{align*}
Let $\tilde{T}_i$ be a sum of $i$ i.i.d.\ random variables with distribution as in \eqref{eqn:F_tilde}. Since
$$
\lambda^i \mathbb{E}\left[\mathbf{1}\left\{T_i\leq \varrho_n+u\right\} e^{-\alpha T_i} \right] = \mathbb{P}(\tilde{T}_i\leq \varrho_n + u )\leq 1,
$$
we have
\begin{align*}
J_1 & \leq \frac{C e^{\alpha u} c_r^2 r^3 (r+1)^2 k_n(z)^5 }{n} \sum_{(p_1,\hdots,p_{r-1})\in(\mathscr{P}_1^n)^{r-1}_{\neq}} \mathbf{1}\{p_1,\hdots,p_{r-1}\text{ do not cross}\} \\
& \hspace{4.75cm} \times  \mathbf{1}\{H(p_1),\hdots,H(p_{r-1})\leq k_n(z) \} \\
& \hspace{4.75cm} \times  \mathbb{P}(p_1,\hdots,p_{r-1}\in\mathcal{P}_1^n) \mathbb{P}(L(p_1),\hdots,L(p_{r-1})\leq \varrho_n + u).
\end{align*}
For $n\to\infty$ the first factor tends to zero, while the sum is bounded due to Lemma \ref{lem:not_crossing}, whence $J_1\to0$ as $n\to\infty$. Moreover, the term
\begin{align*}
J_2 := & C e^{\alpha u} \sum_{(p_1,\hdots,p_r)\in(\mathscr{P}_1^n)^{r}_{\neq}} \mathbf{1}\{p_1,\hdots,p_r\text{ do not cross},\, H(p_1),\hdots,H(p_r)\leq k_n(z) \} \\
& \hspace{2cm} \times \mathbb{P}(p_1,\hdots,p_r\in\mathcal{P}_1^n) \mathbb{P}(L(p_1),\hdots,L(p_{r})\leq \varrho_n + u) \sum_{q=2}^{k_n(z)} \frac{c_r^{q+1} r^{2q} k_n(z)^{2q}}{n^{q-1}}
\end{align*}
also vanishes as $n\to\infty$ since the second factor converges to zero and the remaining sum is again bounded. Now the convergence of $J_1$ and $J_2$ to zero proves the statement of the lemma.
\end{proof}

Let $\mathcal{U}$ be the set of all finite unions of the form $\bigcup_{i=1}^m (a^{(1)}_i,b_i^{(1)}]\times (a_i^{(2)},b_i^{(2)}]$ with $m\in\mathbb{N}$ and $a^{(1)}_i,a^{(2)}_i,b^{(1)}_i,b_i^{(2)}\in\mathbb{R}$ for $i\in\{1,\hdots,m\}$. Next we combine the previous two lemmas to compute the limits of the factorial moments of $\Psi_n^{(a,w)}$, as defined in \eqref{eq:def_Psi_n_a,w}.

\begin{lemma} \label{lem:factorial_moments_prelimit_via_Mr}
Let $a,w\in\mathbb{R}$. For any $A \in \mathcal{U}$ such that $A\subseteq (-\infty,a]\times(-\infty,w]$ and $r\in\mathbb{N}$,
$$
\lim_{n \to \infty} \tau_n^r \E\left[\left(\Psi_n^{(a,w)}(A)\right)_{r} \right] = \big(M^{(r)}\big)^2 \Lambda(A)^r,
$$
with $M^{(r)}$ as in Lemma \ref{lem:recursion_Mr}, and where $\tau_n$ and $\Lambda$ are as defined in \eqref{eq:def_tau_n} and \eqref{eq:def_Lambda}, respectively.
\end{lemma}

\begin{proof}
Note that the assertion follows from
$$
\lim_{n\to\infty} \tau_n^r \E \sum_{(x_1,\hdots,x_r)\in(\Psi^{(a,w)}_n)^r_{\neq}} \mathbf{1}\{ x_1\in A_1,\hdots,x_r\in A_r \} = \big(M^{(r)}\big)^2 \prod_{j=1}^r \Lambda(A_j)
$$
for $A_1,\hdots,A_r\in\mathcal{U}$. Since each of the sets $A_1,\hdots,A_r$ can be written as a disjoint union of Cartesian products of the form $(a^{(1)},b^{(1)}]\times (a^{(2)},b^{(2)}]$ with $a^{(1)}<a^{(2)}\leq a$ and $b^{(1)}<b^{(2)}\leq w$, this is equivalent to
$$
\lim_{n\to\infty} \tau_n^r  \E \sum_{(x_1,\hdots,x_r)\in(\Psi^{(a,w)}_n)^r_{\neq}} \mathbf{1}\{ x_1\in Q_1,\hdots,x_r\in Q_r \} = \big(M^{(r)}\big)^2 \prod_{j=1}^r \Lambda(Q_j)
$$
for all $Q_1,\hdots,Q_r$ such that $Q_i=(\underline{z}_i,\overline{z}_i]\times (\underline{u}_i,\overline{u}_i]$ with $\underline{z}_i<\overline{z}_i\leq a$ and $\underline{u}_i<\overline{u}_i\leq w$ for $i\in\{1,\hdots,r\}$. The latter follows from
\begin{align*}
& \lim_{n\to\infty} \tau_n^r  \E \sum_{(x_1,\hdots,x_r)\in(\Psi^{(a,w)}_n)^r_{\neq}} \mathbf{1}\{ x_1\in (-\infty,z_1]\times(-\infty,u_1],\hdots,x_r\in (-\infty,z_r]\times (-\infty,u_r] \} \\
& = \big(M^{(r)}\big)^2 \prod_{j=1}^r \Lambda((-\infty,z_j]\times(-\infty,u_j])
\end{align*}
for $z_1,\hdots,z_r\in (-\infty,a]$ and $u_1,\hdots,u_r\in (-\infty,w]$, which we show in the sequel. Because of
\begin{align*}
&  \sum_{(x_1,\hdots,x_r)\in(\Psi^{(a,w)}_n)^r_{\neq}} \mathbf{1}\{ x_1\in (-\infty,z_1]\times(-\infty,u_1],\hdots,x_r\in (-\infty,z_r]\times (-\infty,u_r] \} \\
& = \sum_{(p_1,\hdots,p_r)\in (\mathcal{P}_1^n)^r_{\neq}} \mathbf{1}\{ p_1,\hdots,p_r \text{ do not cross and are $(a,w)$-uncrossed}  \} \\
& \hspace{2.5cm} \times \mathbf{1}\{ H(p_i)\leq k_n(z_i) \text{ and } L(p_i)\leq \varrho_n+ u_i \text{ for } i\in[r]\},
\end{align*}
we obtain
\begin{align*}
& \bigg| \tau_n^r \E \sum_{(x_1,\hdots,x_r)\in(\Psi^{(a,w)}_n)^r_{\neq}} \mathbf{1}\{ x_1\in (-\infty,z_1]\times(-\infty,u_1],\hdots,x_r\in (-\infty,z_r]\times (-\infty,u_r] \} \\
& \quad - \tau_n^r  \E \sum_{(p_1,\hdots,p_r)\in (\mathcal{P}_1^n)^r_{\neq}} \mathbf{1}\{ p_1,\hdots,p_r \text{ do not cross}  \} \\
& \hspace{2.5cm} \times \mathbf{1}\{ H(p_i)\leq k_n(z_i) \text{ and } L(p_i)\leq \varrho_n+ u_i \text{ for } i\in[r]\}\bigg| \\
& \leq \tau_n^r  \E \sum_{(p_1,\hdots,p_r)\in (\mathcal{P}_1^n)^r_{\neq}} \mathbf{1}\{ p_1,\hdots,p_r \text{ do not cross and one of them is not $(a,w)$-uncrossed}  \} \\
& \hspace{2.5cm} \times \mathbf{1}\{ H(p_i)\leq k_n(z_i) \text{ and } L(p_i)\leq \varrho_n+ u_i \text{ for } i\in[r]\} \\
& \leq \tau_n^r  \E \sum_{(p_1,\hdots,p_r,\tilde{p})\in (\mathcal{P}_1^n)^{r+1}_{\neq}} \mathbf{1}\{ p_1,\hdots,p_r \text{ do not cross, $\tilde{p}$ crosses } p_1,\hdots,p_r \} \\
& \hspace{2.5cm} \times \mathbf{1}\left\{ H(p_1),\hdots,H(p_r),H(\tilde{p})\leq k_n(a), L(p_1),\hdots,L(p_r),L(\tilde{p})\leq \varrho_n+ w \right\}.
\end{align*}
Since $\tau_n^r$ is bounded and the expectation on the right-hand side vanishes by Lemma \ref{lem:contribution_crossed} as $n\to\infty$, Lemma \ref{lem:not_crossing} yields
\begin{align*}
& \lim_{n\to\infty} \tau_n^r  \E \sum_{(x_1,\hdots,x_r)\in(\Psi^{(a,w)}_n)^r_{\neq}} \mathbf{1}\{ x_1\in (-\infty,z_1]\times(-\infty,u_1],\hdots,x_r\in (-\infty,z_r]\times (-\infty,u_r] \} \\
& = \big( M^{(r)} \big)^2 \prod_{i=1}^r \Lambda((-\infty,z_i]\times(-\infty,u_i]),
\end{align*}
which concludes the proof.
\end{proof}

\subsection{Proof of Theorem \ref{thm:main}}\label{sec:proof_completion}

Now it only remains to combine the findings of the previous subsections to derive the desired point process convergences via the method of moments.

\begin{proof}[Proof of Theorem \ref{thm:main}]
Throughout the proof we consider parts a) and b) simultaneously. We let $\Psi_{\mathbb{R}^2}\sim\operatorname{Poisson}(W_1W_2\mathbb{P}_N\otimes K_{\mathbb{R}})$ and $\Psi_{\vartheta,\R\times M\mathbb{Z}}\sim\operatorname{Poisson}(W_1W_2e^{\alpha\vartheta}\mathbb{P}_N\otimes K_{M\mathbb{Z}})$. We first prove that, for any $A\in\mathcal{U}$,
\begin{equation}\label{eqn:convergence_Psi_n}
\Psi_n(A)\overset{d}{\longrightarrow} \Psi_{\R^2}(A) \quad \text{as} \quad n\to\infty \quad \text{and} \quad \Psi_{n_m}(A)\overset{d}{\longrightarrow} \Psi_{\vartheta,\R\times M\mathbb{Z}}(A) \quad \text{as} \quad m\to\infty
\end{equation}
in the non-arithmetic and in the arithmetic cases, respectively. We choose $a,w\in\mathbb{R}$ such that $A\subseteq (-\infty,a]\times(-\infty,w]$. From Lemma \ref{lem:factorial_moments_prelimit_via_Mr}, Corollary \ref{cor:W_M_r} and Lemma \ref{lem:factorial_moments_limit} with $\Lambda=\mathbb{P}_N\otimes K_{\mathbb{R}}$ we obtain that
\begin{align*}
\lim_{n\to\infty} \E[ (\Psi_n^{(a,w)}(A))_r] & = (M^{(r)})^2 ((\mathbb{P}_N\otimes K_{\mathbb{R}})(A))^r \\
& = (\mathbb{E}[W^r])^2 ((\mathbb{P}_N\otimes K_{\mathbb{R}})(A))^r = \E[ (\Psi_{\R^2}(A))_r]
\end{align*}
for the non-arithmetic case. For the arithmetic case the same findings as well as
$$
\lim_{m\to\infty} \tau_{n_m}^r = \lim_{m\to\infty} n_m^r e^{-r\alpha\varrho_{n_m}} = \lim_{m\to\infty} e^{-r\alpha \left(\varrho_{n_m}-\frac{1}{\alpha}\log(n_m)\right)} = e^{-r\alpha\vartheta}
$$
imply that
\begin{align*}
\lim_{m\to\infty} \E[ (\Psi_{n_m}^{(a,w)}(A))_r] & = \lim_{m\to\infty} \frac{1}{\tau^r_{n_m}} (M^{(r)})^2 ((\mathbb{P}_N\otimes K_{M\mathbb{Z}})(A))^r \\
& = (\mathbb{E}[W^r])^2 (e^{\alpha\vartheta}(\mathbb{P}_N\otimes K_{M\mathbb{Z}})(A))^r = \E[ (\Psi_{\vartheta,\R\times M\mathbb{Z}}(A))_r].
\end{align*}
Since the distributions of $\Psi_{\R^2}(A)$ and $\Psi_{\vartheta,\R\times M\mathbb{Z}}(A)$ are uniquely determined by their moments (see Lemma \ref{lem:moments_define_distribution}), the method of moments yields
\begin{equation}\label{eqn:convergence_approximation}
\Psi_n^{(a,w)}(A)\overset{d}{\longrightarrow} \Psi_{\R^2}(A) \quad \text{as} \quad n\to\infty \quad \text{and} \quad \Psi_{n_m}^{(a,w)}(A)\overset{d}{\longrightarrow} \Psi_{\vartheta,\R\times M\mathbb{Z}}(A) \quad \text{as} \quad m\to\infty
\end{equation}
for the non-arithmetic and the arithmetic cases, respectively. From Lemma \ref{lem:good_paths} we know that
$$
\lim_{n\to\infty}  \E\big[|\Psi_n(A) - \Psi_n^{(a,w)}(A)|\big]  = 0 \quad \text{and} \quad \lim_{m\to\infty}  \E\big[|\Psi_{n_m}(A) - \Psi_{n_m}^{(a,w)}(A)|\big]  = 0,
$$
respectively. This means that $(\Psi_n(A))_{n\in\mathbb{N}}$ and $(\Psi_{n_m}(A))_{m\in\mathbb{N}}$ must have the same limiting distributions as $(\Psi_n^{(a,w)}(A))_{n\in\mathbb{N}}$ and $(\Psi_{n_m}^{(a,w)}(A))_{m\in\mathbb{N}}$. Thus, \eqref{eqn:convergence_approximation} implies \eqref{eqn:convergence_Psi_n}.

By Theorem 23.25 of \cite{k21}, choosing $\mathcal{U}$ as semi-ring and dissecting ring, the one-dimensional convergence in distributions for the numbers of points in sets from $\mathcal{U}$ in \eqref{eqn:convergence_Psi_n} is sufficient for convergence in distribution of the point processes. This completes the proof.
\end{proof}

\subsection{Proof of Theorem \ref{thm:shortest_path}} \label{sec:proof_shortest_path}

We denote by $\mathbf{N}_{\mathbb{R}^2}$ the set of locally finite counting measures on $\mathbb{R}^2$, which we can interpret as point configurations. We let $\mathbf{N}_{\mathbb{R}^2}$ be equipped with the vague topology induced by all maps $\mathbf{N}_{\mathbb{R}^2}\ni \chi \mapsto \int f \, \dint \chi$, where $f:\mathbb{R}^2\to\mathbb{R}$ is a non-negative continuous function with bounded support. This means that a sequence $(\chi_n)_{n\in\mathbb{N}}$ in $\mathbf{N}_{\mathbb{R}^2}$ converges to $\chi\in\mathbf{N}_{\mathbb{R}^2}$ if and only if
$$
\lim_{n\to\infty} \int f \, \dint \chi_n = \int f \, \dint \chi
$$
for all non-negative continuous functions $f:\mathbb{R}^2\to\mathbb{R}$ with bounded support (see e.g.\ \cite[Chapter 23]{k21}). Our goal is to derive Theorem \ref{thm:shortest_path} from Theorem \ref{thm:main} via the continuous mapping theorem. However, the functions of $\Psi_n$ we are interested in are not continuous with respect to the vague topology. For example, we have that $\chi_n=\delta_{(0,0)} + \delta_{(1,-n)}$ converges to $\chi=\delta_{(0,0)}$ as $n\to\infty$, but the smallest second coordinate of $\chi_n$ is $-n$, which does  not converge to $0$, the smallest second coordinate of $\chi$. In order to circumvent such problems, we use some approximation and truncation arguments in the following proof. 

\begin{proof}[Proof of Theorem \ref{thm:shortest_path}]
For $\chi\in \mathbf{N}_{\mathbb{R}^2}$ we define 
$$
l_{\min}(\chi) = \inf_{(x_1,x_2)\in\chi} x_2,
$$
i.e., $l_{\min}(\chi)$ is the smallest second coordinate of a point of $\chi$. Here we use the conventions that $l_{\min}(\chi)=-\infty$ if the second coordinates are unbounded from below and that $l_{\min}(\varnothing)=\infty$. Moreover, we let $p_{\min}(\chi)$ be the number of points of $\chi$ whose second coordinate equals $l_{\min}(\chi)$. Note that $p_{\min}(\chi)=0$ if $l_{\min}(\chi)\in\{-\infty,\infty\}$. We denote by $h(\chi)$ the first coordinates of the points whose second coordinate is $l_{\min}(\chi)$ and think of them as a vector in $\mathbb{R}^{l_{\min}(\chi)}$ ordered by size starting with the smallest element. We put $h(\chi)=\varnothing$ if $p_{\min}(\chi)=0$.

For $u>0$ we define $Q(u)=[-u^2,u^2]\times[-u,u]$ and write $\chi|_{Q(u)}$ for the restriction of $\chi\in\mathbf{N}_{\mathbb{R}^2}$ to $Q(u)$. In the arithmetic case we always assume that $u$ is not a multiple of the span $M$ to ensure that the limiting Cox process does not have points on the boundary of $Q(u)$ almost surely. We have that
\begin{align*}
& \mathbb{P}\left( (l_{\min}(\Psi_n),p_{\min}(\Psi_n),h(\Psi_n))\neq (l_{\min}(\Psi_n|_{Q(u)}),p_{\min}(\Psi_n|_{Q(u)}),h(\Psi_n|_{Q(u)})) \right) \\
& \leq \mathbb{P}(\Psi_n(\mathbb{R}\times (-\infty,u]\setminus Q(u))>0) + \mathbb{P}(\Psi_n(Q(u))=0, \Psi_n\neq\varnothing).
\end{align*}
This statement holds since the functions of the restricted and unrestricted point process coincide if $\Psi_n(Q(u))>0$ and $\Psi_n(\mathbb{R}\times (-\infty,u]\setminus Q(u))=0$ or if $\Psi_n=\varnothing$. From \eqref{eqn:limit_first_moment_II} in Lemma \ref{lem:first_moment} we obtain that
\begin{align*}
\limsup_{n\to\infty} \mathbb{P}(\Psi_n(\mathbb{R}\times (-\infty,u]\setminus Q(u))>0) & \leq \limsup_{n\to\infty} \mathbb{E}[\Psi_n(\mathbb{R}\times (-\infty,u]\setminus Q(u))] \\
& = \limsup_{n\to\infty} \frac{1}{\tau_n} \Lambda(\mathbb{R}\times (-\infty,u]\setminus Q(u)),
\end{align*}
where the limit superior of $(1/\tau_n)_{n\in\mathbb{N}}$ is bounded. As in the proof of Theorem \ref{thm:main} we let $\Psi_{\mathbb{R}^2}\sim\operatorname{Poisson}(W_1W_2\Lambda)$ and $\Psi_{\vartheta,\R\times M\mathbb{Z}}\sim\operatorname{Poisson}(W_1W_2e^{\alpha\vartheta}\Lambda)$. In the non-arithmetic case we have by Theorem \ref{thm:main} and \eqref{eqn:probability_empty_point_process} that
\begin{align*}
\lim_{n\to\infty} \mathbb{P}(\Psi_n(Q(u))=0, \Psi_n\neq\varnothing) & = \lim_{n\to\infty} \mathbb{P}(\Psi_n(Q(u))=0) - \mathbb{P}(\Psi_n=\varnothing) \\
& = \mathbb{P}(\Psi_{\mathbb{R}^2}(Q(u))=0) - \mathbb{P}(W_1,W_2=0) \\
& = \mathbb{E}\left[ \exp\left( - W_1W_2 \Lambda(Q(u)) \right) \right] - \mathbb{P}(W_1,W_2=0) \\
& = \mathbb{E}\left[ \mathbf{1}\{W_1,W_2>0\} \exp\left( - W_1W_2 \Lambda(Q(u)) \right) \right].
\end{align*}
Similarly we obtain for the arithmetic case that
$$
\lim_{m\to\infty} \mathbb{P}(\Psi_{n_m}(Q(u))=0, \Psi_{n_m}\neq\varnothing) = \mathbb{E}\left[ \mathbf{1}\{W_1,W_2>0\} \exp\left( - W_1W_2 e^{\alpha\vartheta} \Lambda(Q(u)) \right) \right].
$$
For $\Psi\in\{\Psi_{\mathbb{R}^2},\Psi_{\vartheta,\R\times M\mathbb{Z}}\}$ we have as for $\Psi_n$ above
\begin{align*}
& \mathbb{P}\left( (l_{\min}(\Psi),p_{\min}(\Psi),h(\Psi))\neq (l_{\min}(\Psi|_{Q(u)}),p_{\min}(\Psi|_{Q(u)}),h(\Psi|_{Q(u)})) \right) \\
& \leq \mathbb{P}(\Psi(\mathbb{R}\times (-\infty,u]\setminus Q(u))>0) + \mathbb{P}(\Psi(Q(u))=0, \Psi\neq\varnothing) \\
& = \mathbb{E}\left[ 1 - \exp\left( W_1W_2c_{\Lambda}\Lambda(\mathbb{R}\times (-\infty,u]\setminus Q(u))\right)\right]\\
&\quad + \mathbb{E}\left[ \mathbf{1}\{W_1,W_2>0\} \exp\left( - c_{\Lambda} W_1W_2 \Lambda(Q(u)) \right) \right]
\end{align*}
with $c_{\Lambda}=1$ for $\Psi=\Psi_{\mathbb{R}^2}$ and $c_{\Lambda}=e^{\alpha\vartheta}$ for $\Psi=\Psi_{\vartheta,\R\times M\mathbb{Z}}$. Together with the observations that $\Lambda(\mathbb{R}\times (-\infty,u]\setminus Q(u))\to 0$ and $\Lambda(Q(u))\to\infty$ as $u\to\infty$, the previous inequalities show that
\begin{align}\label{eqn:truncation_Psi_n}
\nonumber\lim_{u\to\infty} \limsup_{n\to\infty} \mathbb{P}\big( & (l_{\min}(\Psi_n),p_{\min}(\Psi_n),h(\Psi_n))\\ & \neq (l_{\min}(\Psi_n|_{Q(u)}),p_{\min}(\Psi_n|_{Q(u)}),h(\Psi_n|_{Q(u)})) \big)
= 0
\end{align}
in the non-arithmetic case,
\begin{align*}
\lim_{u\to\infty} \limsup_{m\to\infty} \mathbb{P}( & (l_{\min}(\Psi_{n_m}),p_{\min}(\Psi_{n_m}),h(\Psi_{n_m})) \\ & \neq (l_{\min}(\Psi_{n_m}|_{Q(u)}),p_{\min}(\Psi_{n_m}|_{Q(u)}),h(\Psi_{n_m}|_{Q(u)})) ) =0
\end{align*}
in the arithmetic case, and
\begin{equation}\label{eqn:truncation_Psi}
\lim_{u\to\infty} \mathbb{P}\left( (l_{\min}(\Psi),p_{\min}(\Psi),h(\Psi))\neq (l_{\min}(\Psi|_{Q(u)}),p_{\min}(\Psi|_{Q(u)}),h(\Psi|_{Q(u)})) \right)=0.
\end{equation}

Let $u>0$ be fixed. The maps $\chi\mapsto l_{\min}(\chi|_{Q(u)})$, $\chi\mapsto p_{\min}(\chi|_{Q(u)})$ and $\chi\mapsto h(\chi|_{Q(u)})$ are continuous in all $\chi\in\mathbf{N}_{\mathbb{R}^2}$ such that $\chi(\partial Q(u))=0$ and no two points of $\chi$ in $Q(u)$ have the same second coordinate. Since the latter conditions are almost surely satisfied for $\Psi_{\mathbb{R}^2}$, Theorem \ref{thm:main} a) and a generalisation of the continuous mapping theorem (see e.g.\ \cite[Theorem 5.27]{k21}) imply that
\begin{multline*}
(l_{\min}(\Psi_n|_{Q(u)}),p_{\min}(\Psi_n|_{Q(u)}),h(\Psi_n|_{Q(u)}))\\
 \overset{d}{\longrightarrow} (l_{\min}(\Psi_{\mathbb{R}^2}|_{Q(u)}),p_{\min}(\Psi_{\mathbb{R}^2}|_{Q(u)}),h(\Psi_{\mathbb{R}^2}|_{Q(u)}))
\end{multline*}
as $n\to\infty$. Together with \eqref{eqn:truncation_Psi_n} and \eqref{eqn:truncation_Psi}, we obtain
$$
(l_{\min}(\Psi_n),p_{\min}(\Psi_n),h(\Psi_n)) \overset{d}{\longrightarrow} (l_{\min}(\Psi_{\mathbb{R}^2}),p_{\min}(\Psi_{\mathbb{R}^2}),h(\Psi_{\mathbb{R}^2})) \quad \text{as} \quad n\to\infty.
$$
Now the observations that $p_{\min}(\Psi_{\mathbb{R}^2})\in\{0,1\}$ almost surely and that the point of $\Psi_{\mathbb{R}^2}$ with the smallest second coordinate is distributed as in the statement complete the proof of part a).

For the proof of part b) we define $\chi\mapsto \tilde{h}(\chi|_{Q(u)})$ as the first coordinates of all points of $\chi|_{Q(u)}$ such that the second coordinate belongs to $[l_{\min}(\chi),l_{\min}(\chi)+M/3]$. These points are ordered according to size starting with the smallest one. Since the second coordinates of $\Psi_{\vartheta,\R\times M\mathbb{Z}}$ and of $\Psi_n$ for $n\in\mathbb{N}$ are concentrated on multiples of $M$, we have $\tilde{h}(\Psi_n)=h(\Psi_n|_{Q(u)})$ for $n\in\mathbb{N}$ and $\tilde{h}(\Psi_{\vartheta,\R\times M\mathbb{Z}})=h(\Psi_{\vartheta,\R\times M\mathbb{Z}}|_{Q(u)})$. Since $\Psi_{\vartheta,\R\times M\mathbb{Z}}$ is almost surely a continuity point of $\tilde{h}$, we obtain as before from Theorem \ref{thm:main} b) and the continuous mapping theorem
\begin{align*}
& (l_{\min}(\Psi_{n_m}|_{Q(u)}),p_{\min}(\Psi_{n_m}|_{Q(u)}),h(\Psi_{n_m}|_{Q(u)}))\\
& \overset{d}{\longrightarrow} (l_{\min}(\Psi_{\vartheta,\R\times M\mathbb{Z}}|_{Q(u)}),p_{\min}(\Psi_{\mathbb{R}^2}|_{Q(u)}),h(\Psi_{\vartheta,\R\times M\mathbb{Z}}|_{Q(u)}))
\end{align*}
as $m\to\infty$, and
\begin{align*}
 (l_{\min}(\Psi_{n_m}),p_{\min}(\Psi_{n_m}),h(\Psi_{n_m}))
 \overset{d}{\longrightarrow} (l_{\min}(\Psi_{\vartheta,\R\times M\mathbb{Z}}),p_{\min}(\Psi_{\vartheta,\R\times M\mathbb{Z}}),h(\Psi_{\vartheta,\R\times M\mathbb{Z}})) 
\end{align*}
as $m\to\infty$. This implies part b) of the theorem.
\end{proof}

\subsection*{Acknowledgments}
The authors would like to thank the anonymous reviewers for their careful readings of an earlier version of this paper and for suggestions which improved the presentation of the work, as well as the editors for their comments.

\DeclareRobustCommand{\VAN}[3]{#3}
\bibliographystyle{abbrv}
\bibliography{FPPbiblio}

\end{document}